\documentclass[letterpaper]{article} 
\usepackage{aaai2026}
\usepackage{times}  
\usepackage{helvet}  
\usepackage{courier}  
\usepackage[hyphens]{url}  
\usepackage{graphicx} 
\usepackage{natbib}  
\usepackage{caption} 
\usepackage{amsmath}
\usepackage{float}
\usepackage{mathtools}
\mathtoolsset{showonlyrefs}

\usepackage{amsthm}
\usepackage{amssymb}
\usepackage{graphicx}
\usepackage{xcolor}
\usepackage{mathrsfs}

\usepackage{silence}
\usepackage{subcaption}
\makeatletter
\usepackage{comment}

\usepackage{fancyhdr} 
\usepackage{lastpage} 
\usepackage[bb=dsserif]{mathalpha} 
\usepackage{epstopdf} 
\usepackage{amsmath}
\usepackage{algorithm}
\usepackage{algorithmic}
\usepackage{enumerate}
\usepackage{pstricks}

\usepackage{tikz}
\usepackage{pgfplots}
\DeclareCaptionStyle{ruled}{labelfont=normalfont,labelsep=colon,strut=off} 
\usepgfplotslibrary{groupplots,dateplot}
\usetikzlibrary{patterns,shapes.arrows, fadings}
\usetikzlibrary{automata} 
\usetikzlibrary{shapes.geometric}
\usetikzlibrary{calc,shadows}
\usetikzlibrary{positioning} 
\usetikzlibrary{arrows} 
\tikzset{node distance=2.5cm, 
every state/.style={ 
semithick,
fill=gray!10},
initial text={}, 
double distance=2pt, 
every edge/.style={ 
draw,
->,>=stealth', 
auto,
semithick}}
\pgfplotsset{compat=newest}
\newlength\figH
\newlength\figW
\theoremstyle{definition}
\newtheorem{definition}{\protect\definitionname}
\newtheorem{assumption}[definition]{\protect\assumptionname}
\newtheorem{lemma}[definition]{\protect\lemmaname}

\newtheorem{remark}[definition]{Remark}

\theoremstyle{plain}
\newtheorem{theorem}[definition]{\protect\theoremname}
\newtheorem{problem}{Problem}

\numberwithin{definition}{section}

\DeclareMathOperator{\adj}{Adj}

\DeclareMathOperator{\diam}{diam}

\newcommand{\pushright}[1]{\ifmeasuring@#1\else\omit\hfill$\displaystyle#1$\fi\ignorespaces}

\makeatother

\providecommand{\assumptionname}{Assumption}
\providecommand{\definitionname}{Definition}
\providecommand{\lemmaname}{Lemma}
\providecommand{\theoremname}{Theorem}
\providecommand{\propositionname}{Proposition}
\newcommand{\norm}[1]{\left\lVert#1\right\rVert}

\newcommand{\Var}[1]{\text{Var}\left[#1\right]}
\newcommand{\Expectation}[1]{\mathbb{E}\left[#1\right]}

\newcommand{\ind}[1]{\mathbb{I}\left\{#1\right\}}
\newcommand{\Adjacent}[2]{\adj(#1, #2)=1}

\newcommand{\simplex}{\phi}

\newcommand{\SH}{\mathcal{S}_H}

\newcommand{\ones}{\mathbb{1}}

\newcommand{\dbset}{\mathcal{D}}
\newcommand{\seta}{\mathcal{A}}
\newcommand{\setb}{\mathcal{B}}

\newcommand{\action}{\alpha}
\newcommand{\feasible}{\mathcal{F}}
\newcommand{\dual}{\mu}
\newcommand{\prob}[1]{\mathbb{P}\left(#1 \right)}
\newcommand{\dualmax}{\Lambda}
\newcommand{\new}[1]{#1}

\title{Differentially Private Linear Programming: \\ Reduced Sub-Optimality and Guaranteed Constraint Satisfaction}
\author{
    Alexander Benvenuti\textsuperscript{\rm 1}, Brendan Bialy\textsuperscript{\rm 2}, Miriam Dennis\textsuperscript{\rm 2}, Matthew Hale\textsuperscript{\rm 1}\\
}
\affiliations{
    \textsuperscript{\rm 1}School of Electrical and Computer Engineering, Georgia Institute of Technology,
   Atlanta, GA, USA\\
   \textsuperscript{\rm 2}Munitions Directorate, Air Force Research Laboratory, Eglin Air Force Base, FL, USA\\
    abenvenuti3@gatech.edu,
brendan.bialy@us.af.mil, miriam.dennis.1@us.af.mil,
    mhale30@gatech.edu
}

\begin{document}





\maketitle


\begin{abstract}
Linear programming is a fundamental tool in a wide range of decision systems. However, without privacy protections, sharing the solution to a linear program may reveal information about the underlying data used to formulate it, which may be sensitive. Therefore, in this paper we introduce an approach for protecting sensitive data while formulating and solving a linear program. First, we prove that this method perturbs objectives and constraints in a way that makes them differentially private. Then, we show that (i) privatized problems always have solutions, and (ii) their solutions satisfy the constraints in their corresponding original, non-private problems. The latter result solves an open problem in the literature. Next, we analytically bound the expected sub-optimality of solutions that is induced by privacy. Numerical simulations show that, under a typical privacy setup, the solution produced by our method yields a~$65\%$ reduction in sub-optimality compared to the state of the art. 
\end{abstract}

\section{Introduction}
Linear programming is used in a wide range of settings, including resource allocation, power systems, and transportation systems. In many modern systems, user data plays an increasing role in formulating such optimization problems. Sensitive information such as investor data, home power consumption, and travel routes may be use used to formulate these problems~\cite{markowitz1953portfolio, stott1979review}, though sharing the solution to an optimization problem may leak this sensitive data~\cite{hsu2014privately}. As a result, interest has arisen in solving linear programs while both (i) preserving the privacy of the data used in the problem formulation and (ii)  ensuring constraint satisfaction.

In this paper we solve the open problem posed in~\cite{munoz2021private}, namely, the privatization of data that is used to generate constraints when solving linear programs, while also maintaining feasibility of solutions with respect to the original, non-private constraints. For a problem with linear constraints~$Ax \leq b$ and cost~$c^Tx$,
the work in~\cite{munoz2021private} privatized the data that is used to generate~$b$ while also ensuring feasibility with respect to the original constraints. 
Then~\cite{munoz2021private} named it as an open problem to 
simultaneously privatize the data that produces~$A$ and ensure satisfaction 
of the original, non-private constraints. We not only solve this open problem, but in fact go one step further by simultaneously privatizing the data that produces all three {\textemdash} $A$,~$b$, and~$c$ {\textemdash} with guaranteed satisfaction of the original constraints. 

To produce a private linear program, we use differential privacy. Differential privacy is a statistical notion of privacy originally developed to protect entries in databases~\cite{dwork2006calibrating}, and it has seen wide use in the controls~\cite{le2013differentially,cortes2016differential, hawkins2020differentially, yazdani2022differentially}, planning~\cite{chen2023differential, benvenuti2023differentially}, and federated learning~\cite{geyer2017differentially, agarwal2021skellam, chen2022fundamental, noble2022differentially} communities for the strong guarantees that it provides.

We use differential privacy in this work partly because of its immunity to post-processing~\cite{dwork2014algorithmic}, namely that arbitrary computations on private data do not weaken differential privacy. 
We consider linear programs in which the cost~$c^Tx$
and constraint terms~$A$ and~$b$ can all depend on user data, and
we use differential privacy to perturb each of these terms in order to protect the data that is used to generate them. The result is a privacy-preserving linear program. We then solve this optimization problem, which is simply a way of post-processing the privacy-preserving problem. Thus, the solution to the private problem preserves the privacy of the data used to formulate the problem, as do any downstream computations that use that solution. 

\begin{table*}
\centering
\resizebox{\textwidth}{!}{
\begin{tabular}{|c|c|c|c|c|c|c|}
\hline
 & \cite{hsu2014privately} & \cite{cummings2015privacy}& \cite{dvorkin2020differentially}& \cite{munoz2021private} & \cite{benvenuti2024guaranteed}& This work \\
 \hline
 Privatize~$A$ & \checkmark & \checkmark & &  &\checkmark & \checkmark \\
 \hline
Privatize~$b$ & \checkmark & \checkmark & \checkmark & \checkmark & & \checkmark \\
\hline
Privatize~$c$ & \checkmark & \checkmark &  &  & & \checkmark\\
\hline
Satisfy Constraints &  &  & &\checkmark & \checkmark & \checkmark \\
\hline
\end{tabular}}
\caption{Comparison of differentially private linear programming approaches in the literature. 
}
\label{table:related_work}
\end{table*}

As noted in \cite{benvenuti2024guaranteed} and~\cite{munoz2021private}, common privacy mechanisms such as the Gaussian and Laplace mechanisms~\cite{dwork2014algorithmic} add noise with unbounded support. Such mechanisms can perturb constraints by arbitrarily large amounts, which can therefore cause the solution to a privatized problem to be infeasible with respect to the original constraints. Motivated by this challenge, we develop a new matrix truncated Laplace mechanism to privatize the data that produces~$A$, and we use the truncated Laplace mechanism developed in~\cite{munoz2021private, geng2020tight} to privatize the data that produces~$b$. This approach allows us to privatize the constraints such that they only become tighter, thereby ensuring that the solution to the private problem always satisfies the original, non-private constraints. Since the cost does not affect feasibility, we use the unbounded Laplace mechanism  in~\cite{dwork2014algorithmic} to provide privacy for the cost. We then bound the accuracy of our method, which provides users with a tool to calibrate privacy based on its tradeoff with performance. To summarize, our contributions are:
\begin{itemize}
    \item We develop a differential privacy mechanism that simultaneously privatizes all terms in a linear program (Theorem~\ref{thm:privacy}).
    \item We prove that a privatized problem produces a solution that is feasible with respect the constraints in the original, non-private problem (Theorem~\ref{thm:feasible}), which solves an open
    problem in the literature. 
    \item We bound the accuracy of our method, namely the increase/decrease in optimal cost, 
    in terms of privacy parameters (Theorem~\ref{thm:lp_accuracy_e}).
    \item We empirically compare the performance of our method to the state of the art and show that the solution produced by our method yields a~$65\%$ reduction in sub-optimality relative to existing work
    (Section~\ref{sec:sim}).
\end{itemize}

\subsection{Related Work}
There exists substantial previous work on differential privacy in optimization, specifically looking at privacy for objective functions in distributed optimization \cite{huang2015differentially, wang2016differentially,  han2016differentially, nozari2016differentially, dobbe2018customized, lv2020differentially}. We differ from these works because we consider the constraints to also be sensitive, not just objectives. Privacy for optimization with linear constraints has been previously investigated in \cite{hsu2014privately, cummings2015privacy, dvorkin2020differentially, munoz2021private, benvenuti2024guaranteed, kaplan2024differentially}. Both~\cite{hsu2014privately} and \cite{cummings2015privacy} consider differential privacy for both the costs and constraints, but they allow for constraints to be violated, which is unacceptable in many applications, e.g., if constraints encode safety. 
The authors in~\cite{dvorkin2020differentially} analyze privacy for the constant vector 
in equality constraints by reformulating their optimization problem as a stochastic chance-constrained optimization problem. The authors in~\cite{kaplan2024differentially} consider privacy for all the constraints with a focus on maximizing the number of constraints satisfied. However, both of these works still allow for constraint violation. Both~\cite{munoz2021private} and~\cite{benvenuti2024guaranteed} address the problem of privately solving convex optimization problems with linear constraints with guaranteed constraint satisfaction, but~\cite{munoz2021private} only privatizes~$b$ and~\cite{benvenuti2024guaranteed} only privatizes~$A$. We differ because we consider privacy for all components of an LP simultaneously. 
Moreover,~\cite{benvenuti2024guaranteed} privatizes the entries of~$A$ themselves, though here we consider the more general setting of allowing~$A$ and~$b$ to be functions of user data, and we privatize that user data, not the entries of~$A$ and~$b$.  
Table~\ref{table:related_work} summarizes our place in the literature.


\subsection{Notation}\label{subsec:notation}
For $N\in\mathbb{N}$, we use $[N] := \{1, 2, \ldots, N\}$. We use $|\cdot|$ to denote the cardinality of a set and~$A\ominus B$ to denote the symmetric difference between two sets~$A$ and~$B$. 
We use~$\Expectation{X}$ to denote the expectation of a random variable~$X$ and~$\mathcal{L}(\sigma)$ to be a zero-mean Laplace distribution with scale parameter~$\sigma$. We use~$M_{i,j}$ to denote the~$i^{th}j^{th}$ entry of a matrix. Additionally, we use~$\norm{M}_{1, 1} = \sum_{i=1}^{m}\sum_{j=1}^n |M_{i,j}|$ to denote the~$(1, 1)$-norm of a matrix. We use~$\ones^{m\times n}$ to be an~$m\times n$ matrix of all ones and~$[-s\ones^{m \times n}, s\ones^{m \times n}]$ to be an~$mn$-fold Cartesian product of the interval~$[-s, s]$. We use $A\circ B$ as the Hadamard product between matrices~$A$ and~$B$. We write~$\diam(S) = \sup_{s_1, s_2\in S}\norm{s_1-s_2}_2$ for the diameter of a set~$S$.
\section{Preliminaries and Problem Formulation}

\subsection{Linear Programming}\label{subsec:LP}
We consider linear programs (LPs) formed from a database~$D\in\dbset$ taking the form
    \begin{equation}
        \begin{aligned}
        &\begin{aligned}
            \underset{x}{\operatorname{maximize}} &\quad c(D)^Tx
        \end{aligned}
            \\
            &\begin{aligned}
                \operatorname{subject} \operatorname{to } \,\, A(D)x&\leq b(D), \quad x\geq 0, 
            \end{aligned}
        \end{aligned}   
        \tag{P}\label{opt:primal}
    \end{equation} 
    where~$\dbset$ is the set of all possible realizations of the database~$D$,~$c(D)\in\mathbb{R}^n$ is the ``cost vector'', $A(D)\in\mathbb{R}^{m\times n}$ is the ``constraint coefficient matrix'', and~$b(D)\in\mathbb{R}^m$ is the ``constraint vector''.
    We also use the sets~$\seta$ to denote the set of all realizations of~$A(D)$ and~$\setb$ and~$b(D)$ for 
    all~$D \in \dbset$.
We define the feasible region of the LP for a database realization~$D$ as
\begin{equation}\label{eq:fd}
    \feasible(D) = \{x \in \mathbb{R}^n : A(D)x\leq b(D)\}. 
\end{equation}
        \begin{remark}
        We include the constraint~$x\geq 0$ without loss of generality since the constraints in a problem may be reformulated to shift the feasible region to the non-negative orthant without changing the problem. We do this because having strictly positive decision variables allows us have insight into how the feasible region changes when perturbing the constraints, which plays a key role in our feasibility analysis in Section~\ref{sec:priv_constraints}.
    \end{remark}
    \begin{assumption}\label{ass:slate}
        For every~$D\in\dbset$, Problem~\eqref{opt:primal} satisfies Slater's condition.
    \end{assumption}

    
    \begin{remark}\label{rem:equality}
        Assumption~\ref{ass:slate} is common in the optimization literature. 
        Slater's condition says that for the constraints~$Ax\leq b$ there exists a point~$\bar{x}$ such that~$A\bar{x}-b<0$, and thus Assumption~\ref{ass:slate} states that such a point must exist for each realization of the database~$D$. Satisfying Slater's condition implies that the feasible region~$\feasible(D)$ defined by~\eqref{eq:fd} has non-empty interior for all~$D\in\dbset$. If~$\feasible(D)$ has empty interior, then any perturbation to the constraints can produce a private problem whose solution is automatically infeasible with respect to the original, non-private constraints. Thus, such constraints are fundamentally incompatible with privacy.
    We enforce Assumption~\ref{ass:slate} in order to
    only consider problems where it is at least possible
    to attain both privacy and feasibility simultaneously, though
    we still must determine how to do so. 

\begin{assumption}\label{ass:bounded}
    The set~$\dbset$ is bounded and the bounds are publicly available.
\end{assumption} 

Assumption~\ref{ass:bounded} is quite mild since user data may represent physical quantities that do not exceed certain bounds, e.g., with voltages in a power grid, and these can be publicly known without revealing any sensitive user data. 
    
    \end{remark}
    
    
    Problem~\eqref{opt:primal} admits an equivalent dual problem of the form 
    \begin{equation}
        \begin{aligned}
        &\begin{aligned}
            \underset{\dual}{\operatorname{minimize}} &\quad \dual^T b(D)
        \end{aligned}
            \\
            &\begin{aligned}
                \operatorname{subject} \operatorname{to } \,\, \dual^T A(D)&\leq c(D)^T, \quad \dual \geq 0. 
            \end{aligned}
        \end{aligned}   
        \tag{DUAL}\label{opt:dual}
    \end{equation} 
    

    

\subsection{Differential Privacy}
We will provide differential privacy to a database~$D$ by perturbing each component of the LP that it produces, namely~$A(D)$,~$b(D)$, and~$c(D)$. The goal of differential privacy is to make ``similar" pieces of data appear approximately indistinguishable, and the notion of ``similar" is defined by an adjacency relation~\cite{dwork2014algorithmic}. 

\begin{definition}[Adjacency]\label{def:adj2}
    Two databases~$D$ and~$D'$ are said to be ``adjacent" if they differ in at most one entry. If two databases~$D$ and~$D'$ are adjacent, we say~$\adj(D, D')=1$; otherwise 
    we write~$\adj(D, D')=0$.
\end{definition}

To 
make adjacent pieces of data appear approximately indistinguishable,
we implement differential privacy, which is done using a
randomized map called a ``mechanism". 
In its general form, differential privacy
protects a sensitive piece of data~$y$ by
randomizing some function of it, say~$f(y)$. In the case of linear programming, we privatize three functions of the sensitive data~$D$, namely~$A(D)$,~$b(D)$, and~$c(D)$.
\begin{definition}[Differential Privacy; \cite{dwork2014algorithmic}]\label{def:dp}
    Fix a probability space $(\Omega, \mathcal{F}, \mathbb{P})$. 
    Let~$\epsilon>0$ and $\delta\in[0, \frac{1}{2})$ be given. A mechanism $\mathscr{M}:\mathbb{R}^{m\times n}\times\Omega\to\mathbb{R}^{m\times n}$ is ($\epsilon, \delta)$-differentially private if for all $V(D), W(D)\in\mathbb{R}^{m\times n}$ that are adjacent in the sense of Definition \ref{def:adj2}, we have
            $\mathbb{P}[\mathscr{M}(V)\in T] \leq e^\epsilon \mathbb{P}[\mathscr{M}(W)\in T]+\delta$
             for all Borel measurable sets  $T\subseteq\mathbb{R}^{m\times n}$. 
\end{definition}

Since all three components of Problem~\eqref{opt:primal} require privacy, next we state a lemma on how composing private mechanisms affects privacy.

\begin{lemma}[Sequential Composition of Private Mechanisms~\cite{dwork2014algorithmic}]\label{lem:comp}
    For~$i\in[N]$, fix~$\alpha_i\geq 0$ such that~$\sum_{i=1}^N \alpha_i = 1$. Let~$\mathscr{M}_i:\dbset\to \mathcal{R}_i$ for~$i\in[N]$ be an~$(\new{\alpha_i}\epsilon, \new{\alpha_i}\delta)$-differentially private mechanism. If~$\mathscr{M}_{[N]}:\dbset \to \prod_{i=1}^N \mathcal{R}_i$ is defined to be~$\mathscr{M}_{[N]}(D) = \left(\mathscr{M}_1(D), \ldots, \mathscr{M}_N(D)\right)$ then~$\mathscr{M}_{[N]}$ is~$(\epsilon, \delta)$-differentially private.
\end{lemma}

\new{We refer to the~$\alpha_i$'s as the ``privacy budget allocation", since they divide~$\epsilon$ and~$\delta$ among the privacy mechanisms.} Lemma~\ref{lem:comp} implies that an algorithm containing individual privatizations of~$A(D)$,~$b(D)$, and~$c(D)$ is itself differentially private with parameters equal to the sum of the privacy parameters from each individual privatization. This property allows us to form a linear program composed of each privatized quantity and ensure that forming that program is differentially private. Next we state a lemma that solving such an optimization problem also preserves the privacy of the underlying database~$D$.

\begin{lemma}[Immunity to Post-Processing; {\cite{dwork2014algorithmic}}]\label{lem:arbitrary}
    Let~$\mathscr{M}:\mathbb{R}^{m\times n}\times \Omega \to \mathbb{R}^{m\times n}$ be an~$(\epsilon, \delta)$-differentially private mechanism. Let~$h:\mathbb{R}^{m\times n}\to\mathbb{R}^{p\times q}$ be an arbitrary mapping. Then the composition~$h\circ\mathscr{M}:\mathbb{R}^{m\times n}\to\mathbb{R}^{p\times q}$ is~$(\epsilon, \delta)$-differentially private.
\end{lemma}


Since solving an optimization problem is a form of post-processing, Lemma~\ref{lem:arbitrary} implies that the solution to an $(\epsilon, \delta)$-differentially private optimization problem is also $(\epsilon, \delta)$-differentially private, allowing the solution to a privatized form of Problem~\eqref{opt:primal} to be shared without harming the privacy of~$D$.

\subsection{Problem Statements}
Consider Problem~\eqref{opt:primal}. 
Computing~$x^*$ without any protections depends on the underlying sensitive database~$D$, and thus computing and using~$x^*$ can reveal information about~$D$. Therefore, we seek to develop a framework for solving problems in the form of Problem~\eqref{opt:primal} that preserves the privacy of~$D$ while still satisfying the constraints in Problem~\eqref{opt:primal}. 
This will be done by solving the following problems.

\begin{problem}\label{prob:private}
    Develop a privacy mechanism to privatize~$D$ when computing each component of a linear program (namely, $A(D)$, $b(D)$, and~$c(D)$).
\end{problem}

\begin{problem}\label{prob:feasible}
    Prove that a solution to the privately generated optimization problem also satisfies the constraints of the original, non-private problem.
\end{problem}

\begin{problem}\label{prob:accurate}
    Bound the sub-optimality in solutions that is induced by the privacy mechanism in terms of the privacy parameters~$\epsilon$ and~$\delta$.
\end{problem}

\section{Private Constraints}\label{sec:priv_constraints}
In this section, we solve Problems~\ref{prob:private} and~\ref{prob:feasible}. 
Specifically, we perturb the matrix~$A(D)$, the vector~$b(D)$, and the vector~$c(D)$ in order to privatize~$D$.
    Since~$A(D)$, ~$b(D)$, and~$c(D)$ all have different properties and requirements to preserve feasibility, we use separate privacy mechanisms for each. 
    We begin with privacy for the matrix~$A(D)$.
\subsection{Privacy for~$A(D)$}
\begin{definition}\label{def:l11_sens}
    The $L_{1,1}$-sensitivity of a function $f:\dbset\to\mathbb{R}^{m\times n}$ is
    \begin{equation}
        \Delta_{1,1} f = \sup_{D, D': \Adjacent{D}{D'}} \norm{f(D) - f(D')}_{1, 1}.
        \end{equation}
\end{definition}


Next, we extend the definition of the Truncated Laplace Distribution in~\cite{munoz2021private} to matrix-valued draws.

\begin{lemma}[Matrix-Variate Truncated Laplace Mechanism]
\label{lem:matrix_lap}
    Let privacy parameters~$\epsilon>0$ and~$\delta\in(0, \frac{1}{2}]$ and sensitivity~$\Delta_{1,1}A$ be given.
    The Matrix-Variate Truncated Laplace Mechanism takes a matrix-valued function of sensitive data~$F(y)\in\mathbb{R}^{m\times n}$ as input and outputs the private approximation of~$F(y)$, denoted $\tilde{F}(y) = F(y) + Z \in \mathbb{R}^{m\times n}$, 
    where ${Z_{i,j}\sim\mathcal{L}_T(\sigma_A, \mathcal{S}_A)}$ for all~$i\in[m]$ and~$j\in[n]$. Here,~$\mathcal{L}_T(\sigma_A, \mathcal{S}_A)$ is the scalar truncated Laplace distribution with density
        $f(Z_{i,j}) = \frac{1}{\zeta}\exp\left(-\frac{1}{\sigma_A}|Z_{i,j}|\right),$
    where~$\mathcal{S}_A := [-s_A, s_A]$ and the values of~$s_A$ and~$-s_A$ are bounds on the private outputs such that~$Z_{i,j}\in\mathcal{S}_A$. We define~$\zeta=\mathbb{P}(Z_{i,j}\leq |s_A|)$, and~$\sigma_A$ is the scale parameter of the distribution. The Matrix-Variate Truncated Laplace Mechanism is~$(\epsilon, \delta)$-differentially private 
    if~$\sigma_A \geq \frac{\Delta_{1,1}A}{\epsilon}$ and
        $s_A = \frac{\Delta_{1,1}A}{\epsilon}\log\left(\frac{mn(\exp(\epsilon)-1)}{\delta}+1\right)$.
        %
\end{lemma}
\begin{proof}
See Technical Appendix~\ref{app:lap_mat}.    
\end{proof}
We apply Lemma~\ref{lem:matrix_lap} to the entire constraint matrix~$A(D)$ to produce a differentially private constraint matrix~$\bar{A}$. 
If some entry~$A(D)_{i,j}$ is identically zero for all~$D\in\dbset$, then that zero entry may represent that there is no physical relationship between a decision variable and a constraint. For example, in a smart power grid system, one home's power consumption may not influence its neighbor's power consumption. Given their practical relevance, we wish to preserve such properties when privacy is implemented. 
To ensure that \new{identically} zero-valued entries 
in~$A(D)$ 
\new{(which we refer to as ``non-sensitive'' entries)} remain unchanged by privacy, we set
\begin{equation}\label{eq:barA}
    \bar{A} = A(D)+(s_A\ones^{m\times n}+Z)\circ\ind{A(D)\neq 0}, 
\end{equation}
where~$\ind{A\new{(D)}\neq 0}\in\mathbb{R}^{m\times n}$ 
is a matrix of ones and zeros where~$\ind{A\new{(D)}\neq 0}_{i,j} = 1$ if \new{there exists some~$D\in\mathcal{D}$ such that~$A(D)_{i,j}\neq 0$} and 
$\ind{A\new{(D)}\neq 0}_{i,j} = 0$ if~$A\new{(D)}_{i,j} = 0$ \new{ for all~$D\in\mathcal{D}$}.
We add~$s_A\ones^{m\times n}$ in~\eqref{eq:barA} to ensure that the coefficients in~$\bar{A}$ can only become larger than
they were in~$A(D)$, thus tightening the constraints to promote feasibility of private solutions with respect to the original, non-private constraints.



\subsection{Privacy for~$b(D)$}
To enforce privacy for~$b(D)$, we leverage the approach used in \cite{munoz2021private}, which we restate here for completeness. We begin by defining the sensitivity of a vector-valued function.
\begin{definition}\label{def:l1_sens}
    The $\ell_{1}$-sensitivity of a function $f:\dbset\to\mathbb{R}^{m}$ is
        $\Delta_{1} f = \sup_{D, D': \Adjacent{D}{D'}} \norm{f(D) - f(D')}_{1}$.
\end{definition}

We use the multivariate Truncated Laplace Mechanism to enforce privacy for the constraint vector~$b(D)$.

\begin{lemma}[Multivariate Truncated Laplace Mechanism~\cite{munoz2021private, geng2020tight}]\label{lem:vec_lap}
    Let privacy parameters~$\epsilon>0$ and~$\delta\in(0, \frac{1}{2}]$ and sensitivity~$\Delta_{1}b$ be given.
    The Truncated Laplace Mechanism takes a function of sensitive data~$f(y)\in\mathbb{R}^{m}$ as input and outputs a private approximation of~$f(y)$, denoted~$\tilde{f}(y) = f(y) + z \in \mathbb{R}^{m}$, where~$z_{i} \in \mathcal{S}_b$, with~$\mathcal{S}_b := [-s_b, s_b]$, and ${z_{i}\sim\mathcal{L}_T(\sigma_b, \mathcal{S}_b)}$ for all~$i\in[m]$.
     The multivariate truncated Laplace mechanism is~$(\epsilon, \delta)$-differentially private if~$\sigma_b\geq \frac{\Delta_{1}b}{\epsilon}$ and
        $s_b = \frac{\Delta_{1}b}{\epsilon}\log\left(\frac{m(\exp(\epsilon)-1)}{\delta}+1\right) $. 
        
\end{lemma}

We apply Lemma~\ref{lem:vec_lap} to~$b(D)$ to obtain
\begin{equation}\label{eq:barB}
    \bar{b} = b(D)-s_b\ones^{m}+z_b
\end{equation}
as the privatized constraint vector, where~$z_b$ is the noise added to enforce privacy from Lemma~\ref{lem:vec_lap}. By subtracting~$s_b\ones^{m}$, we ensure that each entry in the constraint vector becomes smaller, thereby tightening each constraint to promote feasibility. 

\subsection{Privacy for~$c(D)$}
To enforce privacy for~$c(D)$, we use the standard Laplace mechanism, which we define next.
\begin{lemma}[Laplace Mechanism; \cite{dwork2014algorithmic}]\label{lem:lap_mech}
    Let $\Delta_1f>0$ and $\epsilon>0$ be given, and fix the adjacency relation from Definition \ref{def:adj2}. The Laplace mechanism takes sensitive data $f(y)\in\mathbb{R}^n$ as input and outputs private data
        $\tilde{f}(y)=f(y)+z$,
    where $z\sim \mathcal{L}(\sigma)$. The Laplace mechanism is $(\epsilon, 0)$-differentially private if
        $\sigma\geq \frac{\Delta_1 f}{\epsilon}.$
\end{lemma}

Similar to the \new{identically} zero entries of~$A(D)$, an \new{identically} zero, \new{or non-sensitive zero}, entry in~$c(D)$ encodes the fact that a decision variable does not impact the cost, and thus to preserve that structure we privatize only the \new{non-sensitive} zero elements in~$c(D)$. Let~$c(D)^0$ denote the vector of \new{sensitive} entries of~$c(D)$. To produce a private cost function, we compute
    $\tilde{c}^0 = c(D)^0+z_c$, 
where~$z_c\sim\mathcal{L}(\sigma_c)$ is the noise added using Lemma~\ref{lem:lap_mech} to enforce privacy. We then form~$\tilde{c}$ by replacing the \new{sensitive} entries in~$c(D)$ with the corresponding private 
entries of~$\tilde{c}^0$.
Since changing the cost function does not impact feasibility,~$\tilde{c}^0$ requires no post-processing and may be used as-is. 
\subsection{Guaranteeing Feasibility}
Along with privacy, 
we must also enforce feasibility. 
In order for the privately obtained solution~$\tilde{x}^*$ 
to satisfy the constraints of the non-private problem
(namely Problem~(P)), 
it is clear 
that the two problems must have
at least one feasible point in common. 
Ensuring that this is always true  
thus leads to the following assumption.

\begin{assumption}[Perturbed Feasibility]\label{ass:feast}
    The set~$S=\bigcap_{D\subseteq\dbset}\left\{x: A(D)x\leq b(D)\right\}$ is not empty.
\end{assumption}

In words, Assumption~\ref{ass:feast} says that there must exist at least
one point that satisfies the constraints produced by
every realization of the database~$D$. 
With Assumption~\ref{ass:feast}, we post-process~$\bar{A}$ from~\eqref{eq:barA}  and~$\bar{b}$ from~\eqref{eq:barB} 
according to
\begin{equation} \label{eq:abarpost}
\!\!\tilde{A}_{i,j} = \min\Big\{\bar{A}_{i,j}, \sup_{D \in \dbset} A(D)_{i,j} \Big\} \text{ for all } i\!\in\![m],\! j\!\in\![n]
\end{equation}
and
\begin{equation} \label{eq:bbarpost}
\tilde{b}_{i} = \max\Big\{\bar{b}_{i}, \inf_{D \in \dbset} b(D)_{i} \Big\} \text{ for all } i\in[m]. 
\end{equation}
For~$\tilde{A}_{i,j}$, we 
do so for each~$(i, j)$ such that~$A_{i,j}$ is non-zero.
For~$\tilde{b}_i$, we do so for all~$i$.
The outputs of these computations are the private
constraint coefficient matrix~$\tilde{A}$ and private constraint vector~$\tilde{b}$. 

\begin{remark}
    Taking the minimum in~\eqref{eq:abarpost} ensures that each entry in~$\tilde{A}$ appears in some~$A\in\seta$ and taking the maximum in~\eqref{eq:bbarpost} ensures that each entry in~$\tilde{b}$ appears in some~$b\in\setb$. 
    The supremum and infumum are finite since~$\dbset$ is bounded, and 
    computing them maintains privacy since~$\dbset$ 
    does not depend on sensitive information according to Assumption~\ref{ass:bounded}. 
\end{remark}

With this privacy implementation, we will solve the optimization problem
    \begin{equation}
        \begin{aligned}
        &\begin{aligned}
            \underset{x}{\operatorname{maximize}} &\quad \tilde{c}^Tx
        \end{aligned}
            \\
            &\begin{aligned}
                \operatorname{subject} \operatorname{to } \,\,\tilde{A}x&\leq \tilde{b}, \quad x\geq 0. 
            \end{aligned}
        \end{aligned}   
        \tag{DP-P}\label{opt:DP}
    \end{equation} 

Algorithm~\ref{algo:solve} provides a unified
overview of our approach. \new{Note that Problem~\eqref{opt:DP} may be solved via any algorithm, and thus Algorithm~\ref{algo:solve} does not introduce any additional computational complexity compared to solving Problem~\eqref{opt:primal}.}

\begin{algorithm}[t]
    \caption{Privately Solving Linear Programs}
    \label{algo:solve}
    \begin{algorithmic}[1]
    \STATE \textbf{Inputs}: Problem~\eqref{opt:primal}, $\epsilon$, $\delta$, $\Delta_{1,1}A$, $\Delta_1b$, $\Delta_1c$, $\alpha_A$, $\alpha_b$,\! $\alpha_c$ \\
    \STATE \textbf{Outputs}: Privacy-preserving solution $\tilde{x}^*$ \\
    \vspace{1mm}
    \STATE Set $\sigma_A = \frac{\Delta_{1,1}A}{\new{\alpha_A}\epsilon}$ \\
    \STATE Set $\sigma_b = \frac{\Delta_1b}{\new{\alpha_b}\epsilon}$ \\
    \STATE Set $\sigma_c = \frac{\Delta_1c}{\new{\alpha_c}\epsilon}$ \\
    \STATE Compute the support for the constraint coefficient matrix 
    $s_A = \frac{\Delta_{1,1}A}{\new{\alpha_A}\epsilon}\log\left(\frac{2nm(\exp(\new{\alpha_A}\epsilon)-1)}{\delta}+1\right)$ \\
    \STATE Compute the support for the constraint vector 
    $s_b = \frac{\Delta_{1}b}{\new{\alpha_b}\epsilon}\log\left(\frac{2m(\exp(\new{\alpha_b}\epsilon)-1)}{\delta}+1\right)$ \\
    \STATE Generate $\bar{A}$ using~\eqref{eq:barA}
    \\
    \STATE Generate $\bar{b}$ using~\eqref{eq:barB}
    \\
    \STATE Post-process~$\bar{A}$ using~\eqref{eq:abarpost} \\
        \STATE Post-process~$\bar{b}$ using~\eqref{eq:bbarpost} \\
    \STATE Compute~$\tilde{c}^0 = c(D)^0+z_c$\\
    \STATE Form $\tilde{c}$ by replacing each non-zero entry
    in~$c$ with its corresponding entry of~$\tilde{c}^0$\\
    \STATE Solve Problem~\eqref{opt:DP} (via any algorithm) to find~$\tilde{x}^*$
    \end{algorithmic}
\end{algorithm}
\begin{figure*}[t]
\begin{equation}\label{eq:rho}
    \rho = 
        \begin{cases}
            \bigg( 2m\left(\frac{\Delta_{1}b}{\new{\alpha_b}\epsilon}\right)^2 + ms_b^2 + 2s_b\chi\sum_{i=1}^m\left(n^0_i s_A\right) +\dualmax^2\left(2\left(\frac{\Delta_{1}b}{\new{\alpha_b}\epsilon}\right)^2 + s_b^2\right)+\\\chi^2\sum_{i=1}^m\left(2n^0_i\left(\frac{\Delta_{1,1}A}{\new{\alpha_A}\epsilon}\right)^2 + (n^0_i s_A)^2\right) +
         2n^{0, c} \left(\frac{\Delta_{1}c}{\new{\alpha_c}\epsilon}\right)^2 +\\ m\dualmax^2\sum_{j=1}^n \left(2m^0_j\left(\frac{\Delta_{1,1}A}{\new{\alpha_A}\epsilon}\right)^{\new{2}} + (m^0_j s_A)^2\right)+ 2\chi^2 n^{0, c} \left(\frac{\Delta_{1}c}{\new{\alpha_c}\epsilon}\right)^2  \bigg)^{\frac{1}{2}} & \parbox[t]{4.5cm}{\vspace{-1.7cm}if $\tilde{A}_{i,j} = A(D)_{i,j}+(s_A+Z_{i,j})\ind{A\neq 0}_{i,j}$ and $\tilde{b}_i = b(D)_i-s_b+z_{b_i}$ for all $i,j$}\\

         \norm{\begin{bmatrix}
            \sqrt{n}\|(A(D)-\hat{A})\|_F\chi \\
            \sqrt{m}\|(A(D)-\hat{A})^T\|_F\dualmax\\
             2\sqrt{n}\frac{\Delta_{1}c}{\new{\alpha_c}\epsilon}\chi
        \end{bmatrix}}_2
        + \norm{\begin{bmatrix}
            \|(b(D)-\hat{b})\|_2 \\
           2\frac{\Delta_{1}c}{\new{\alpha_c}\epsilon}\\
            \sqrt{m}\|b(D)-\hat{b}\|_2\dualmax
        \end{bmatrix}}_2
         & \text{otherwise}
        \end{cases}
    \end{equation}
\hrulefill
\end{figure*}
\begin{remark}\label{rem:modify}
    Algorithm~\ref{algo:solve} presents our approach in the case where every component of the LP depends on the sensitive database; however, one can amend Algorithm~\ref{algo:solve} if only a subset of these components depends on the sensitive database. For example, if only~$A(D)$ and~$c(D)$ depend on the database, and the constraint vector~$b$ does not, then one can omit steps 4, 7, 9, and 11 from Algorithm~\ref{algo:solve}, and \new{choose~$\alpha_A> 0$ and~$\alpha_c>0$ such that~$\alpha_A+\alpha_c=1$}. Doing so will yield a more accurate result while still guaranteeing privacy for the database-dependent quantities. 
\end{remark}

\subsection{Characterizing Privacy}
Next we prove that Algorithm~\ref{algo:solve} is differentially private. 

\begin{theorem}[Solution to Problem~$1$]\label{thm:privacy}
    Let privacy parameters~$\epsilon>0$ and~$\delta\in(0, \frac{1}{2}]$, sensitivities~$\Delta_{1,1}A$,~$\Delta_1b$, and~$\Delta_1c$\new{, and privacy budget allocations~$\alpha_i$ for~$i\in\{A, b, c\}$}
    be given. Let Assumptions~\ref{ass:slate},~\ref{ass:bounded}, and~\ref{ass:feast} hold. Then Algorithm~\ref{algo:solve} keeps the database~$D$ 
    $(\epsilon, \delta)$-differentially private.
\end{theorem}

\begin{proof}
    See Technical Appendix~\ref{app:thm_privacy}
\end{proof}

Theorem~\ref{thm:privacy} allows us to privatize each component of the linear program individually to generate an overall~$(\epsilon, \delta)$-differentially private LP. The solution generated by solving~\eqref{opt:DP} then can be shared without harming privacy.

\begin{theorem}[Solution to Problem~2]\label{thm:feasible}
    Let privacy parameters~$\epsilon>0$ and~$\delta\in[0, \frac{1}{2})$ be given, and let Assumptions~\ref{ass:slate},~\ref{ass:bounded}, and~\ref{ass:feast} hold. Then Problem~\eqref{opt:DP} is guaranteed to have a solution, and that solution is guaranteed to satisfy the original, non-privatized constraints in Problem~\eqref{opt:primal}. 
\end{theorem}
\begin{proof}
    See \new{Technical} Appendix~\ref{app:thm_feasible}.
\end{proof}


Theorem~\ref{thm:feasible} guarantees that Algorithm~\ref{algo:solve} produces a feasible LP. Since all of the constraints are tightened by privacy, and the solution to Problem~\eqref{opt:DP} always exists, that solution is guaranteed to satisfy the original, non-private constraints. The conclusion of~\cite{munoz2021private} identifies the privatization of~$A(D)$ with guaranteed constraint satisfaction as an open problem, and thus Theorems~\ref{thm:privacy} and~\ref{thm:feasible} not only solve this open problem but present, to the best of the authors' knowledge, the only private linear programming approach which can simultaneously privatize~$A(D)$ and~$b(D)$ while guaranteeing satisfaction of the original, non-private constraints.

\section{Accuracy}
In this section, we solve Problem~\ref{prob:accurate}
and bound the expected sub-optimality that is induced by privacy. 
This bound depends on (i) the largest feasible solution and the largest possible norm of a dual variable, whether it is a solution or not, 
(ii) the realization of the LP components at the boundary of~$\dbset$, and (iii) the ``closeness" of the private and non-private optimization problems in a way that we make precise.

For (i) and (ii), we define the following quantities: 
\begin{align}
    \chi&= \max_{D\in\dbset, j\in [N]} \bar{x}(D)_j^* \label{eq:chi}\\
    \dualmax &= \max_{D\in\dbset}\frac{c(D)^T\eta - c(D)^T\omega}{\min\limits_{j\in[m]} -A(D)_j\eta+b(D)_j} \label{eq:dualmax}
    \end{align}
    \begin{align}
    \hat{A} &= \left[\sup_{D\in\dbset} A(D)_{i,j}\right]_{i\in[m], j\in[n]}\label{eq:ahat}\\
    \hat{b} &= \left[\sup_{D\in\dbset} b(D)_{i}\right]_{i\in[m]}, \label{eq:bhat}
\end{align}
    
    
    
    
\noindent where~$\eta$ is a solution to Problem~\eqref{opt:primal}
and~$\omega$ is a Slater point for Problem~\eqref{opt:primal}. 
For (iii), we 
use Corollary~$3.1$ from~\cite{robinson1973bounds}, which we state formally in the Technical Appendix, Section~\ref{sec:sup_lem} as Lemma~\ref{lem:perturbation}.
Lemma~\ref{lem:perturbation} then forms the basis for our accuracy result, which we state next.
\begin{theorem}[Solution to Problem~\ref{prob:accurate}]\label{thm:lp_accuracy_e}
    Fix privacy parameters~$\epsilon>0$ and~$\delta\in[0, \frac{1}{2})$, and let the sensitivities~$\Delta_{1,1}A$,~$\Delta_1b$, and~$\Delta_1c$, \new{and privacy budget allocations~$\alpha_i$ for~$i\in\{A, b, c\}$} be given. Let Assumptions~\ref{ass:slate},~\ref{ass:bounded}, and~\ref{ass:feast} hold. Let~$x^*$ be the solution to Problem~\eqref{opt:primal} and let~$\tilde{x}^*$ be the solution to Problem~\eqref{opt:DP}. 
    Let~$H(G, C)$ be the Hoffman constant, as defined in~\cite{robinson1973bounds}, associated with Problem~\eqref{opt:primal}. 
    Then
         $\Expectation{c(D)^Tx^*-c(D)^T\tilde{x}^*}\leq \norm{c(D)}_2H(G, C)\rho$, 
    where~$\rho$ is defined in~\eqref{eq:rho}. 
\end{theorem}

\begin{proof}
See \new{Technical} Appendix~\ref{app:accuracy}.    
\end{proof}

\begin{remark}
    Given an LP in the form of Problem~\eqref{opt:primal}, the Hoffman constant~$H(G, C)$ is well-defined and always exists. 
    Computing exact Hoffman constants is known to be NP-Hard~\cite{pena2018algorithm}, though a variety of upper bounds
    and efficient approximation algorithms for them exist, and any one of
    them can be used in conjunction with Theorem~\ref{thm:lp_accuracy_e}. 
    A full exposition is beyond the scope of this article, and we refer the reader to~\cite{pena2018algorithm,pena2021new,hoffman1952approximate} and references therein for an extended discussion.
\end{remark}

The accuracy guarantee of Theorem~\ref{thm:lp_accuracy_e} enables users to trade off the worst-case \new{average} sub-optimality induced by privacy in order to design the parameters~$\epsilon$ and~$\delta$. \new{Additionally, we find that the order of~$\rho$ is~$\mathcal{O}(\log(\frac{1}{\delta})\epsilon^{-\frac{1}{2}})$ in terms of the privacy parameters and~$\mathcal{O}(nm)$ in terms of the problem parameters, which implies linear growth of the error in terms of the privacy parameters in the worst case. However, in practice, we find virtually no increase in error with increasing problem size, which is shown empirically in Section~\ref{sec:sim}}. \new{Next, we bound the magnitude of the fluctuations around the mean.}
\begin{theorem}\label{cor:concentration}
            Let the conditions of Theorem 4.1 hold. Let~$R = \norm{x^*-\tilde{x}^*}_2$. Then,
            \begin{equation}
             \mathbb{P}\left(R - \Expectation{R}\geq \diam(\mathcal{F}(D))\sqrt{\log(1/t)/2}\right)\geq 1-t.
            \end{equation}
\end{theorem}
\begin{proof}
See Technical Appendix~\ref{app:concentration}.  
\end{proof}
\new{Theorem~\ref{cor:concentration} indicates that the fluctuations of the error about the mean error depends on the size of the original, non-private feasible region.}
\section{Numerical Simulation}\label{sec:sim}
In this section, we present simulations
on the internet advertising setting described in~\cite{munoz2021private}, which we restate here for completeness. \new{See Technical Appendix~\ref{app:cmdp} for additional simulations.} In this setting, the pages of a website are partitioned into~$N$ groups,
and group~$i \in [N]$ receives~$n_i$ unique visitors. 
The database~$D$ is the confidential business information from advertisers, such as market research on the products being advertised.
For a group of~$M$ advertisers, for each~$j\in[M]$, advertiser~$j$ lists a price~$p_{ij}(D)$ that they are willing to pay per unique visit, and a budget~$b(D)_j$ that that they are willing to spend on advertising. 
This setting yields the optimization problem
\begin{equation}
        \begin{aligned}
        &\begin{aligned}
            \underset{x\geq 0}{\operatorname{maximize}} &\quad \sum_{i\in[N]}\sum_{j\in[M]}p_{ij}(D)x_{ij}
        \end{aligned}
            \\
            &\begin{aligned}
                \operatorname{subject} \operatorname{to } \,\,&\sum_{j\in[M]} x_{ij}\leq n_i \quad \text{for }i\in[N]\\
                &\sum_{i\in[N]} p_{ij}(D) x_{ij} \leq b(D)_j \quad \text{for }j\in[M].
            \end{aligned}
        \end{aligned}   
        \tag{EX1}\label{opt:EX1}
    \end{equation} 
We consider two scenarios: (i) where only~$p_{ij}(D)$ requires privacy for all~$i,j$ and (ii) where both~$p_{ij}(D)$ and~$b(D)$ require privacy. \new{For both scenarios,~$\alpha_i = 1/3$ for all~$i$. See Technical Appendix~\ref{subsec:priv_budget} for results with varying~$\alpha_i$.}

For case (i), we utilize Algorithm~\ref{algo:solve} modified in the manner detailed in Remark~\ref{rem:modify} to provide privacy for just~$p_{ij}(D)$, and we compare the quality of the solution (i.e., the loss in revenue) and the fraction of the constraints violated (i.e., how often advertisers go over budget) to that of~\cite{hsu2014privately}. We note that~\cite{hsu2014privately} has known issues with privacy leakages, as pointed out in~\cite{munoz2021private}. 
Specifically,~\cite{hsu2014privately} recommends scaling the problem by the~$\ell_1$ norm of the optimal solution, though doing so will alter the sensitivity of the problem, and no analysis is provided on how to compute the scaled problem's sensitivity. Additionally, the multiplicative weights algorithm used by~\cite{hsu2014differential} in problems with private constraints, such as in their Algorithm 5, may fail to converge in practice when noise draws with strong privacy are too large. Knowing these issues,~\cite{hsu2014privately} still presents the closest work to ours. 

In case (i), we analyze sub-optimality, \new{$\Expectation{(c(D)^Tx^*-c(D)^T\tilde{x}^*)/(c(D)^Tx^*)}$} under varying levels of privacy. We set~$N = 10$ and~$M = 5$. Each~$p_{ij}(D)$ is~$0$ with probability~$0.2$ and is drawn uniformly from~$[0, 1]$ with probability~$0.8$. We also set~$b_i = 10^7$ for all~$i\in[N]$ and~$n_j = 10^7$ for all~$j\in[M]$. The performance loss for case (i) is shown in Figure~\ref{fig:example_1} for~$\epsilon \in [\new{0.25}, 2]$ and~$\delta = 0.1$. Our work maintains zero constraint violation, as guaranteed by Theorem~\ref{thm:feasible}, while~\cite{hsu2014privately}
violates constraints at every value of~$\epsilon$, \new{up to~$51\%$ of constraints at~$\epsilon = 0.25$}.

For case (ii), we use Algorithm~\ref{algo:solve} without any modifications \new{using the same problem parameters as case (i)}. There is no other work to the authors' knowledge that can keep both~$p_{ij}(D)$ and~$b(D)$ private simultaneously. 
\new{However, we still present comparisons to~\cite{munoz2021private}. The work in~\cite{munoz2021private} can only keep~$b(D)$ private in the constraints, which means that it leaks private information about~$p_{ij}(D)$ by not keeping~$A(D)$ private, 
but we include this comparison to quantify how performance is affected by providing privacy to~$A(D)$ in addition to~$c(D)$ and~$b(D)$.}

\begin{figure}
    \centering
        \input{Figures/performance_hsu}
            
    \caption{
    Performance loss 
    with varying privacy strength. 
    Combining Algorithm 5 with a privatized objective from~\cite{hsu2014privately} leads to constraint violation for all~$\epsilon\in[0.25, 2]$.
    High constraint violation allows the solution to give the appearance of superior performance; however, such a solution leads to significant violation of some advertisers' budgets, which is unacceptable. Even when allowing this constraint violation, the solution produced by~\cite{hsu2014privately} still yields worse performance than that of Remark~\ref{rem:modify} and Algorithm~\ref{algo:solve}. 
    We also compare to~\cite{munoz2021private}, and we emphasize that Munoz incurs lower sub-optimality
    because it privatizes only~$b(D)$ in the constraints, while Algorithm~1 is used to privatize both~$A(D)$ and~$b(D)$.
    The approach in~\cite{munoz2021private} only incurs~$0.5\%$ sub-optimality at~$\epsilon = 2$,
    while the approach in Algorithm~1 incurs roughly~$20$\% sub-optimality, which indicates
    that privacy for~$A(D)$ induces~$19.5\%$ additional sub-optimality. 
    }
    \label{fig:example_1}
\end{figure} 


The performance loss for case (ii) is shown in Figure~\ref{fig:example_1}. In case~(ii), 
when more quantities are kept private, i.e., both~$p_{ij}(D)$ and~$b(D)$, Algorithm~\ref{algo:solve} still out-performs~\cite{hsu2014privately}, which highlights our method's improvement over the state of the art. In addition, we see at most a~$6\%$ difference in the performance of our method 
between cases (i) and (ii) (which occurs when~$\epsilon = 2$), indicating only minor performance loss with additional private quantities in an LP. Similarly, we find a~$20\%$ increase in sub-optimality between Algorithm~\ref{algo:solve} and~\cite{munoz2021private} at~$\epsilon = 2$, indicating modest increases in sup-optimality when privatizing both~$p_{ij}(D)$ and~$b(D)$ in the constraints as opposed to only~$b(D)$.




Next, we analyze how performance varies with increasing problem size, 
namely increasing~$M$ 
shown in Figure~\ref{fig:example_2}. \new{We note that increasing~$M$ increases both the number of variables and the number of sensitive constraints.} Since we have shown that our mechanism never violates the constraints while the work in~\cite{hsu2014privately} does, we instead focus our comparisons on performance. We fix~$\epsilon = 1$,~$\delta = 0.1$,~$N = 20$, and a randomly drawn~$p_{ij}(D)$, where~$p_{ij}(D)$ is~$0$ with probability~$0.2$ and is drawn uniformly from~$[0, 1]$ with probability~$0.8$. 
We simulate \new{100} samples for each~\new{$M\in\{5, 6, \ldots, 100\}$}. 
Since Algorithm~5 in~\cite{hsu2014privately} runs for more iterations on problems with more decision variables,
they see an increase in solution quality with increasing~$M$, though in exchange for a significant increase in computation time. We see only a~$10\%$ increase sub-optimality for a~$10$ fold increase in~$M$ without any change in computational complexity due to privacy, highlighting the scalability of our work.
\new{Additionally when privatizing~$b(D)$ with increasing~$M$, the performance of Algorithm~1 is 
virtually identical to that of the method from~\cite{munoz2021private}, 
which highlights both methods' ability to perform at scale.}

\begin{figure}
    \centering
            \input{Figures/performance_size_N}
            
    \caption{Performance loss with varying~$M$. As the number of variables increases, Algorithm~5 in~\cite{hsu2014privately} allows their solver to run for more iterations, leading to improvement in accuracy, though this leads to a dramatic increase in computation time. In Algorithm~\ref{algo:solve}, we see only an~$11\%$ decrease in optimal revenue with a~$10\times$ increase in problem size, 
    going from~$13.3\%$ sub-optimality with~$M = 10$ to~$24\%$ sub-optimality at~$M=100$. Performance remains roughly constant with increasing number of constraints for Algorithm~\ref{algo:solve} \new{and~\cite{munoz2021private}}, while Algorithm~5 in~\cite{hsu2014privately} steadily improves in performance but still has much higher sub-optimality.
    }
    \label{fig:example_2}
\end{figure}

\section{Conclusion}
We presented a method for simultaneously keeping the constraints and costs private in linear programming. We showed that this method is differentially private and that it always produces a feasible solution with respect to the original, non-private constraints. Future work will focus on guaranteed constraint satisfaction while privatizing nonlinear and stochastic constraints.
\section*{Acknowledgements}
This work was partially supported by AFRL under grant FA8651-23-F-A008, 
NSF under CAREER grant 2422260 and Graduate Research Fellowship under grant DGE-2039655, ONR under grant N00014-24-1-2432, and AFOSR under grant FA9550-19-1-0169. Any opinions, findings and conclusions or recommendations expressed in this material are those of the authors and do not necessarily reflect the views of sponsoring agencies.

\bibliography{main}

\section*{Checklist}

\new{
\begin{enumerate}
\item This paper:
\begin{enumerate}
    \item Includes a conceptual outline and/or pseudocode description of AI methods introduced (NA)
    \item Clearly delineates statements that are opinions, hypothesis, and speculation from objective facts and results (yes)
    \item Provides well marked pedagogical references for less-familiare readers to gain background necessary to replicate the paper (yes)
\end{enumerate}
\item Does this paper make theoretical contributions? (yes)
\begin{enumerate}
    \item All assumptions and restrictions are stated clearly and formally. (yes)
    \item All novel claims are stated formally (e.g., in theorem statements). (yes)
    \item Proofs of all novel claims are included. (yes)
    \item Proof sketches or intuitions are given for complex and/or novel results. (yes)
    \item Appropriate citations to theoretical tools used are given. (yes)
    \item All theoretical claims are demonstrated empirically to hold. (partial)
    \item All experimental code used to eliminate or disprove claims is included. (NA)
\end{enumerate}
\item Does this paper rely on one or more datasets? (no)
\item Does this paper include computational experiments? (yes). \\
See Technical Appendix~\ref{ap:checklist} for explanations the following questions.
\begin{enumerate}
    \item This paper states the number and range of values tried per (hyper-) parameter during development of the paper, along with the criterion used for selecting the final parameter setting. (NA)
    \item Any code required for pre-processing data is included in the appendix. (no)
    \item All source code required for conducting and analyzing the experiments is included in a code appendix. (no)
    \item All source code required for conducting and analyzing the experiments will be made publicly available upon publication of the paper with a license that allows free usage for research purposes. (no)
    \item All source code implementing new methods have comments detailing the implementation, with references to the paper where each step comes from (no)
    \item If an algorithm depends on randomness, then the method used for setting seeds is described in a way sufficient to allow replication of results. (no)
    \item This paper specifies the computing infrastructure used for running experiments (hardware and software), including GPU/CPU models; amount of memory; operating system; names and versions of relevant software libraries and frameworks. (no)
    \item This paper formally describes evaluation metrics used and explains the motivation for choosing these metrics. (yes)
    \item This paper states the number of algorithm runs used to compute each reported result. (yes)
    \item Analysis of experiments goes beyond single-dimensional summaries of performance (e.g., average; median) to include measures of variation, confidence, or other distributional information. (no)
    \item The significance of any improvement or decrease in performance is judged using appropriate statistical tests (e.g., Wilcoxon signed-rank). (no)
    \item This paper lists all final (hyper-)parameters used for each model/algorithm in the paper’s experiments. (yes)
\end{enumerate}
\end{enumerate}
}

\newpage 
\onecolumn
\appendix 
\section{Hoffman Constant Bound}\label{sec:sup_lem}

The following lemma will make reference to the dual problem to Problem~\eqref{opt:DP}. For completeness, we explicitly write it out here:
   \begin{equation}
        \begin{aligned}
        &\begin{aligned}
            \underset{\dual}{\operatorname{minimize}} &\quad \mu^T\tilde{b}
        \end{aligned}
            \\
            &\begin{aligned}
                \operatorname{subject} \operatorname{to } \,\,\mu^T\tilde{A}&\leq \tilde{c}^T, \quad \mu\geq 0. 
            \end{aligned}
        \end{aligned}   
        \tag{DP-DUAL}\label{opt:DPD}
    \end{equation} 
    \begin{lemma}[Perturbation Bound; \cite{robinson1973bounds}]\label{lem:perturbation}
    Let Problems~\eqref{opt:primal} and~\eqref{opt:dual} be given. 
    Let~$x^*$ be the solution to Problem~\eqref{opt:primal},
    let~$\mu^*$ be the solution to Problem~\eqref{opt:dual},
    let~$\tilde{x}^*$ be the solution to Problem~\eqref{opt:DP}, and
    let~$\tilde{\mu}^*$ be the solution to Problem~\eqref{opt:DPD}. 
    Let~$G\in\mathbb{R}^{m\times n}$ and~$C\in\mathbb{R}^{p\times n}$ be defined as 
    \begin{equation}
        G = \begin{bmatrix}
         -A(D) & 0\\
         0 & A(D)^T\\
         -I & 0\\
         0 & -I
        \end{bmatrix}, \quad C = \begin{bmatrix}
            c(D)^T & -b(D)^T
        \end{bmatrix}.
    \end{equation}
    Then
    \begin{equation}\label{eq:hoffman_ineq}
        \norm{\begin{bmatrix}
            x^*\\ \dual^*
        \end{bmatrix} - \begin{bmatrix}
            \tilde{x}^*\\ \tilde{\dual}^*
        \end{bmatrix}}_2\leq 
        H_{2,2}(G, C)\norm{\begin{bmatrix}
            [(A-\tilde{A})\tilde{x}^* - (b-\tilde{b})]^{-}\\ [(A-\tilde{A})^T\tilde{\dual}^* - (c-\tilde{c})]^{+}\\
            (c-\tilde{c})\tilde{x}^* - (b-\tilde{b})\tilde{\dual}^*
        \end{bmatrix}}_2,
    \end{equation}
    where $[\cdot]^+$ is the projection onto the non-negative orthant
    of~$\mathbb{R}^m$,~$[\cdot]^-$ is the projection onto the non-positive orthant of~$\mathbb{R}^m$,
    and~$H_{2, 2}(G, C)$ is the Hoffman constant of~$G$ and~$C$, defined as the smallest constant which makes~\eqref{eq:hoffman_ineq} true. 
   
\end{lemma}

\section{Omitted Proofs}
\subsection{Proof of Lemma~\ref{lem:matrix_lap}}\label{app:lap_mat}
The Matrix-Variate Truncated Laplace Mechanism is~$(\epsilon, \delta)$-differentially private if the following two conditions hold:
    \begin{enumerate}
        \item Let~$U\in[A(D)-s_A\ones^{m\times n}, A(D)+s_A\ones^{m\times n}]\cap [A(D')-s_A\ones^{m\times n}, A(D')+s_A\ones^{m\times n}]$ be the matrices in the support of the Matrix-Variate Truncated Laplace Mechanism applied to~$A(D)$ and~$A(D')$. Let~$f_D(u)$ and~$f_{D'}(u)$ be the resulting probability density functions from applying the Matrix-Variate Truncated Laplace Mechanism to~$A(D)$ and~$A(D')$, respectively. Then~$f_D(u)\leq e^{\epsilon}f_{D'}(u)$. 
        \item Let~$V\in[A(D)-s_A\ones^{m\times n}, A(D)+s_A\ones^{m\times n}]\setminus [A(D')-s_A\ones^{m\times n}, A(D')+s_A\ones^{m\times n}]$ be the matrices in the support of the Matrix-Variate Truncated Laplace Mechanism applied to~$A(D)$ but not in the support of the matrix-variate Laplace mechanism applied to~$A(D')$. Then~$\prob{\left[A(D)+Z\right]\in V}\leq \delta$.
    \end{enumerate}

    We start with Condition 1. Since each entry in~$u\in U$ is an independent draw from the truncated Laplace mechanism, we have that
    \begin{equation}
        f_D(u) = 
        \begin{cases}
            \prod_{i=1}^{n}\prod_{j=1}^m\frac{1}{\zeta_{i,j}}\exp\left(-\frac{\epsilon |u_{i,j} +A(D)_{i,j}|}{\Delta_{1,1}A}\right) & u\in[A(D)-s_A\ones^{m\times n}, A(D)+s_A\ones^{m\times n}]\\
            0 & \textnormal{otherwise}
        \end{cases},
    \end{equation}
    where~$\zeta_{i,j}$ is a normalizing constant.
    Since~$u\in[A(D)-s_A\ones^{m\times n}, A(D)+s_A\ones^{m\times n}]\cap [A(D')-s_A\ones^{m\times n}, A(D')+s_A\ones^{m\times n}]$, we have that
    \begin{align}
        \frac{f_D(u)}{f_{D'}(u)} &= \prod_{i=1}^{n}\prod_{j=1}^m\frac{\exp\left(-\frac{\epsilon |u_{i,j} +A(D)_{i,j}|}{\Delta_{1,1}A}\right)}{\exp\left(-\frac{\epsilon |u_{i,j} +A(D')_{i,j}|}{\Delta_{1,1}A}\right)}\\
        & = \prod_{i=1}^{n}\prod_{j=1}^m\exp\left(\frac{\epsilon |u_{i,j} +A(D')_{i,j}|-\epsilon|u_{i,j} +A(D)_{i,j}|}{\Delta_{1,1}A}\right).\label{eq:bound_triangle_reverse}
    \end{align}
    We then bound~\eqref{eq:bound_triangle_reverse} using the reverse triangle inequality to obtain
    \begin{align}
        \prod_{i=1}^{n}\prod_{j=1}^m\exp\left(\frac{\epsilon |u_{i,j} +A(D')_{i,j}|-\epsilon|u_{i,j} +A(D)_{i,j}|}{\Delta_{1,1}A}\right) &\leq \prod_{i=1}^{n}\prod_{j=1}^m\exp\left(\frac{\epsilon |A(D)_{i,j}-A(D')_{i,j}|}{\Delta_{1,1}A}\right) \\
        &=
        \exp\left(\frac{\epsilon \sum_{i=1}^{n}\sum_{j=1}^m|A(D)_{i,j}-A(D')_{i,j}|}{\Delta_{1,1}A}\right).
    \end{align}
    The term~$\sum_{i=1}^{n}\sum_{j=1}^m|A(D)_{i,j}-A(D')_{i,j}|$ is simply the~$(1, 1)$-norm of~$A(D)-A(D')$, which we bound using the sensitivity relation from Definition~\ref{def:l11_sens} to obtain
    \begin{equation}
        \exp\left(\frac{\epsilon \sum_{i=1}^{n}\sum_{j=1}^m|A(D)_{i,j}-A(D')_{i,j}|}{\Delta_{1,1}A}\right)\leq \exp\left(\frac{\epsilon \Delta_{1,1}A}{\Delta_{1,1}A}\right) = e^{\epsilon}.
    \end{equation}
    Therefore we have that
    \begin{equation}
        \frac{f_D(u)}{f_{D'}(u)}\leq e^{\epsilon},
    \end{equation}
    which shows that Condition 1 holds. 
    
    For Condition 2, if~$A(D)+Z>0\in V$, then for some~$i\in[m]$ and~$j\in[n]$ such that~$A(D)_{i,j}>0$, 
    we have that either~$A(D)_{i,j}+Z_{i,j}<A(D')_{i,j}-s_A$ or~$A(D)_{i,j}+Z_{i,j}>A(D')_{i,j}+s_A$. As a result, we know that either~$Z_{i,j}<\Delta_{1,1}A -s_A$ or~$Z_{i,j}> s_A-\Delta_{1,1}A$. 
    Given the probability density function~$f$ from Lemma~\ref{lem:matrix_lap}, and noting that there are~$nm$ entries in~$A(D)$, we can compute~$\prob{\left[A(D)+Z\right]\in V}$ by
    \begin{multline}\label{eq:sub_zeta}
        nm\left(\int_{-s_A}^{-s_A+\Delta_{1,1}A} f(z) dz + \int_{s_A-\Delta_{1,1}A}^{s_A} f(z) dz\right) = \\ \frac{nm}{\zeta}\left(\int_{-s_A}^{-s_A+\Delta_{1,1}A} \exp\left(-\frac{|z|\epsilon}{\Delta_{1,1}A} \right) dz + \int_{s_A-\Delta_{1,1}A}^{s_A} \exp\left(-\frac{|z|\epsilon}{\Delta_{1,1}A} \right) dz\right).
    \end{multline}
    Since~$\zeta=\mathbb{P}(z_i\leq |s_A|) = 2\Delta_{1,1}A(1-e^{-\frac{\epsilon s_A}{\Delta_{1,1}A}})\frac{1}{\epsilon}$, evaluating the integrals and substituting the expression for~$\zeta$ into~\eqref{eq:sub_zeta}, we find
    \begin{equation}
        \frac{nm}{\zeta}\left(\int_{-s_A}^{-s_A+\Delta_{1,1}A} \exp\left(-\frac{|z|\epsilon}{\Delta_{1,1}A} \right) dz + \int_{s_A-\Delta_{1,1}A}^{s_A} \exp\left(-\frac{|z|\epsilon}{\Delta_{1,1}A} \right) dz\right) = \frac{nm(e^{\epsilon}-1)}{1-e^{-\frac{\epsilon s_A}{\Delta_{1,1}A}}} = \delta,
    \end{equation}
    where the final equality follows from the definition of~$s_A$ from Lemma~\ref{lem:matrix_lap}. We then have that~$\prob{\left[A(D)+Z\right]\in V}\leq \delta$, which satisfies Condition 2. 
    
    Next, we will show how Conditions 1 and 2 imply~$(\epsilon, \delta)$-differential privacy. Let~$M\subseteq [A(D)-s_A\ones^{m\times n}, A(D)+s_A\ones^{m\times n}]$ be any set of matrices in the support of the Matrix-Variate Truncated Laplace mechanism applied to~$A(D)$. Let~$M_0 = M\cap [A(D')-s_A\ones^{m\times n}, A(D')+s_A\ones^{m\times n}]$ be the set of matrices in~$M$ that are also in the support of the Matrix-Variate Truncated Laplace Mechanism applied to~$A(D')$. Additionally, let~$M_1 = M\setminus [A(D')-s_A\ones^{m\times n}, A(D')+s_A\ones^{m\times n}]$ be the set of matrices in~$M$ that are not in the support of the Matrix-Variate Truncated Laplace Mechanism applied to~$A(D')$. Noting that~$M_0$ and~$M_1$ completely partition~$M$, we have that
     \begin{equation}
        \prob{A(D)+Z\in M} = \prob{A(D)+Z\in M_0} + \prob{A(D)+Z\in M_1}.
    \end{equation}
    From the definition of~$M_1$ it follows then that~$M_1\subseteq V$, and thus
    \begin{equation}
        \prob{A(D)+Z\in M_0} + \prob{A(D)+Z\in M_1} \leq \prob{A(D)+Z\in M_0} + \prob{A(D)+Z\in V}.
    \end{equation}
   Expressing these probabilities in terms of density functions, we have that
   \begin{equation}\label{eq:need_cond}
       \prob{A(D)+Z\in M_0} + \prob{A(D)+Z\in V} = \int_{M_0} f_D(u) du + \int_{V} f_D(u)du.
   \end{equation}
    From Condition 1, 
    \begin{equation}\label{eq:cond1}
        \int_{M_0} f_D(u) du \leq e^{\epsilon}\int_{M_0} f_{D'}(u) du
    \end{equation}
    and from Condition 2,
    \begin{equation}\label{eq:cond2}
         \int_{V} f_D(u)du \leq \delta.
    \end{equation}
    Substituting the inequalities in~\eqref{eq:cond1} and~\eqref{eq:cond2} into~\eqref{eq:need_cond} yields
    \begin{equation}
        \int_{M_0} f_D(u) du + \int_{V} f_D(u)du \leq e^{\epsilon}\int_{M_0} f_{D'}(u) du + \delta.
    \end{equation}
    Noting that~$\int_{M_0} f_{D'}(u) du = \prob{A(D')+Z\in M}$, we have that
    \begin{equation}
         \prob{A(D)+Z\in M}\leq e^{\epsilon}\prob{A(D')+Z\in M} + \delta,
    \end{equation}
    which is the definition of~$(\epsilon, \delta)$-differential privacy from Definition~\ref{def:dp}. Thus, Conditions 1 and 2 imply~$(\epsilon, \delta)$-differential privacy. 
    Since Conditions 1 and 2 have been satisfied, the statement in Lemma~\ref{lem:matrix_lap} follows.
    
    \hfill$\square$

    \subsection{Proof of Theorem~\ref{thm:privacy}}\label{app:thm_privacy}
We see that~$\sigma_A$,~$s_A$,~$\sigma_B$, and~$s_B$ are all computed using~$\new{\alpha_i}\epsilon$
    and~$\new{\alpha_i}\delta$ \new{ for~$i\in\{A, b\}$}, and that~$\sigma_c$ is computed using~$\new{\alpha_c}\epsilon$. Thus the computations of~$\tilde{A}$ and~$\tilde{b}$ each keep~$D$~$(\new{\alpha_i}\epsilon, \new{\alpha_i}\delta)$-differentially private \new{for~$i\in\{A, b\}$} from Lemmas~\ref{lem:matrix_lap} and \ref{lem:vec_lap}, and the computation of~$\tilde{c}$ keeps~$D$~$(\new{\alpha_c}\epsilon, 0)$-differentially private from Lemma~\ref{lem:lap_mech}. Additionally, since Problem~\eqref{opt:DP} is a composition of~$(\new{\alpha_i}\epsilon, \new{\alpha_i}\delta)$-differentially private mechanisms for~$i\in\{A, b\}$ and an~$(\new{\alpha_c}\epsilon, 0)$-differentially private mechanism, \new{ and since~$\alpha_A+\alpha_b+\alpha_c = 1$}, we conclude that Problem~\eqref{opt:DP} itself keeps~$D$~$(\epsilon, \delta)$-differentially private via Lemma~\ref{lem:comp}. It then follows that Algorithm~\ref{algo:solve} provides~$D$ with~$(\epsilon, \delta)$-differentially privacy since solving Problem~\eqref{opt:DP} is simply post-processing of private data, and thus is~$(\epsilon, \delta)$-differentially private according to Lemma~\ref{lem:arbitrary}.\hfill$\square$
    \subsection{Proof of Theorem~\ref{thm:feasible}}\label{app:thm_feasible}
    Since~$\tilde{b}\in\setb$ as a result of~\eqref{eq:bbarpost}, we will fix an arbitrary~$\tilde{b}\in\setb$ and show that the solution satisfies the original, non-private constraints, 
    which implies that the same is true for all~$\tilde{b}$. We follow a similar approach to Theorem 2 in~\cite{benvenuti2024guaranteed}, and we restate some steps for completeness.
    
    By definition, the constraint matrix~$\tilde{A}$ is component-wise less than or equal to the matrix~$\hat{A}$ in which~$\hat{A}_{i,j}
     = \sup_{D\in\dbset} A_{i,j}(D)$ for all~$i \in [m]$ and~$j \in [n]$. 
    Since~$x\geq 0$ and the vector~$\tilde{b}$ 
    is fixed, we have that the set~$\{x:\hat{A}x\leq \tilde{b}\}$ is contained in~$\{x:\tilde{A}x\leq \tilde{b}\}$ due to the fact that $\tilde{A}_{i,j}\leq \sup_{D\in\dbset} A_{i,j}(D)$.
    Thus, we know that $\{x:\tilde{A}x\leq \tilde{b}\}\supseteq \{x:\hat{A}x\leq \tilde{b}\}$. 
    
    We will now show that the sets~$\bigcap_{D\subseteq\dbset}\{x:A(D)x\leq \tilde{b}\}$ and~$\{x:\hat{A}x\leq \tilde{b}\}$ are equal. 
    For any~$x$ in the first set, 
    we know that $A_i(D)\cdot x\leq \tilde{b}_i$ 
    for all~$A(D) \in \dbset$. 
    By definition of the supremum, it follows then that~$\sup_{D\in\dbset} (A_i(D)\cdot x)\leq \tilde{b}_i$ for all~$i \in [m]$, 
    and therefore~$\hat{A}x\leq \tilde{b}$. As a result, if~$x\in\bigcap_{D\subseteq\dbset}\{z:A(D)z\leq \tilde{b}\}$, then~$x\in\{z:\hat{A}z\leq \tilde{b}\}$. 
    We now show that the reverse is true. If~$\hat{A}x\leq \tilde{b}$, then~$A(D)x\leq \tilde{b}$ for all~$A(D) \in \seta$
     by definition of the supremum. 
     Thus, if~$x\in\{z:\hat{A}z\leq \tilde{b}\}$, then~$x\in\bigcap_{D\subseteq\dbset}\{z:A(D)z\leq \tilde{b}\}$. 
     
    Since we have~$\{x:\tilde{A}x\leq \tilde{b}\}\supseteq \{x:\hat{A}x\leq \tilde{b}\}$ and~$\{x:\hat{A}x\leq \tilde{b}\} = \bigcap_{D\in\dbset}\{x:A(D)x\leq \tilde{b}\}$, we know that~$\{x:\tilde{A}x\leq \tilde{b}\}\supseteq\bigcap_{D\in\dbset}\{x:A(D)x\leq \tilde{b}\}$. From Assumption~\ref{ass:feast}, the set~$\bigcap_{D\in\dbset}\{x:A(D)x\leq \tilde{b}\}$ is non-empty, and therefore the set~$\{x:\tilde{A}x\leq \tilde{b}\}$ is non-empty, and thus it yields a feasible optimization problem. 
    Since this is true for any~$\tilde{b}\in\setb$, it follows that the set~$\{x:\tilde{A}x\leq \tilde{b}\}$ is non-empty and thus it yields a feasible optimization problem for all~$\tilde{b}\in\mathcal{B}$. Since this approach only tightens the constraints, any~$x\in\{z:\tilde{A}z\leq \tilde{b}\}$ will also satisfy the original, non-private constraints, as desired. 
    
    \hfill$\square$

    \subsection{Proof of Theorem~\ref{thm:lp_accuracy_e}}\label{app:accuracy}
    Starting with~$\Expectation{c(D)^Tx^*-c(D)^T\tilde{x}^*} $, we  factor out~$c(D)^T$ and use the Cauchy-Schwarz inequality to bound~$\Expectation{c(D)^Tx^*-c(D)^T\tilde{x}^*} $ as
    \begin{equation}\label{eq:upper_bound_with_cat}
        \Expectation{c(D)^Tx^*-c(D)^T\tilde{x}^*} \leq \norm{c(D)}_2\Expectation{\norm{x^* - \tilde{x}^*}_2}.
    \end{equation}
    We then concatenate~$x^*$ with the solution to Problems~\eqref{opt:dual}  and~$\tilde{x}^*$ with the solution to Problem~\eqref{opt:DPD}, to upper bound~\eqref{eq:upper_bound_with_cat} by
    \begin{equation}\label{eq:sub_ex_here}
        \norm{c(D)}_2\Expectation{\norm{x^* - \tilde{x}^*}_2} \leq \norm{c(D)}_2\Expectation{\norm{\begin{bmatrix}
            x^*\\ \dual^*
        \end{bmatrix} - \begin{bmatrix}
            \tilde{x}^*\\ \tilde{\dual}^*
        \end{bmatrix}}_2}.
    \end{equation}
    We then focus on bounding~$\Expectation{\norm{\begin{bmatrix}
            x^*\\ \dual^*
        \end{bmatrix} - \begin{bmatrix}
            \tilde{x}^*\\ \tilde{\dual}^*
        \end{bmatrix}}_2}$. From Lemma~\ref{lem:perturbation} 
        we have
    \begin{equation}\label{eq:initial_acc}
        \Expectation{\norm{\begin{bmatrix}
            x^*\\ \dual^*
        \end{bmatrix} - \begin{bmatrix}
            \tilde{x}^*\\ \tilde{\dual}^*
        \end{bmatrix}}_2}\leq H_{2,2}(G, C)\Expectation{\norm{\begin{bmatrix}
            [(A-\tilde{A})\tilde{x}^* - (b-\tilde{b})]^{-}\\ [(A-\tilde{A})^T\tilde{\dual}^* - (c-\tilde{c})]^{+}\\
            (c-\tilde{c})^T\tilde{x}^* - (b-\tilde{b})^T\tilde{\dual}^*
        \end{bmatrix}}_2},
    \end{equation}
     where 
     \begin{equation}
     \tilde{A}_{i,j} = \min\Big\{A(D)_{i,j}+(s_A+Z_{i,j})\circ\ind{A(D)\neq 0}_{i,j}, \sup_{D\in\dbset} A(D)_{i,j}\Big\},
     \end{equation}
     \begin{equation}
     \tilde{b}_i  =  \max\Big\{b(D)_i-s_b+z_{b_i},\inf_{D\in\dbset} b(D)_i\Big\}, 
     \end{equation}
     and~$\tilde{c}_i = z^0_{c_i} +c(D)_i$, where~$z_c^0\in\mathbb{R}^n$ is the noise added to~$c(D)$; 
     the~$i^{th}$ entry of~$z_c^0$ is~$z_{c,i}^0 = z_{c,i}$ if~$c(D)_i \neq 0$ and it is~$0$ otherwise. 
     We will analyze two cases: (i)~$\tilde{A}_{i,j} = A(D)_{i,j}+(s_A+Z_{i,j})\circ\ind{A(D)\neq 0}_{i,j}$ and~$\tilde{b}_i = b(D)_i-s_b+z_{b_i}$ for all~$i,j$, 
     and (ii) there exist indices~$i, j$ where $\tilde{A}_{i,j} = \sup_{D\in\dbset} A(D)_{i,j}$ or~$\tilde{b}_i = \inf_{D\in\dbset} b(D)_i$. Beginning with case (i), using the non-expansive property of the projections onto the non-negative and non-positive orthants, and substituting in the quantities~$\tilde{A}$,~$\tilde{b}$, and~$\tilde{c}$, we have that 
    \begin{multline}\label{eq:sub_y}
        H_{2,2}(G, C)\Expectation{\norm{\begin{bmatrix}
            [(A-\tilde{A})\tilde{x}^* - (b-\tilde{b})]^{-}\\ [(A-\tilde{A})^T\tilde{\dual}^* - (c-\tilde{c})]^{+}\\
            (c-\tilde{c})^T\tilde{x}^* - (b-\tilde{b})^T\tilde{\dual}^*
        \end{bmatrix}}_2} \leq \\ H_{2,2}(G, C)\Expectation{\norm{\begin{bmatrix}
            -\Big((s_A\ones^{n\times m}+Z)\circ\ind{A(D)\neq 0}\Big)\tilde{x}^* + (-s_b\ones^{m}+z_{b})\\ \Big( {-(s_A}\ones^{m\times n}+Z)\circ\ind{A(D)\neq 0}\Big)^T\tilde{\dual}^* + z^0_c\\
            -(z_c^0)^T\tilde{x}^* + (-s_b\ones^m+z_{b})^T\tilde{\dual}^*
        \end{bmatrix}}_2}.
    \end{multline} 
    Next, for~$i \in [m]$ and~$j \in [n]$, 
    we define~$Y_{i,j} = (s_A+Z_{i,j})\circ\ind{A(D)\neq 0}_{i,j}$ and~$y_{b_i} = -s_b+z_{b_i}$.
    From Lemma~\ref{lem:matrix_lap} we see that~$Y_{i,j}\sim \mathcal{L}_T(\frac{\Delta_{1,1}A}{\new{\alpha_A}\epsilon},[0, 2s_A])$ for all~$Y_{i, j}\neq 0$,
    and from Lemma~\ref{lem:vec_lap} we see that~$y_{b_i}\sim \mathcal{L}_T(\frac{\Delta_{1}b}{\new{\alpha_b}\epsilon},[-2s_b, 0])$. 
    We substitute these into ~\eqref{eq:sub_y} to find
    \begin{multline}
        H_{2,2}(G, C)\Expectation{\norm{\begin{bmatrix}
            -\Big((s_A\ones^{n\times m}+Z)\circ\ind{A(D)\neq 0}\Big)\tilde{x}^* + (-s_b\ones^{m}+z_{b})\\ \Big( {-(s_A}\ones^{m\times n}+Z)\circ\ind{A(D)\neq 0}\Big)^T\tilde{\dual}^* + z^0_c\\
            -(z_c^0)^T\tilde{x}^* + (-s_b\ones^m+z_{b})^T\tilde{\dual}^*
        \end{bmatrix}}_2} =\\ H_{2,2}(G, C)\Expectation{\norm{\begin{bmatrix}
            -Y\tilde{x}^* + y_b\\  -Y^T\tilde{\dual}^* + z_c^0\\
            -(z_c^0)^T\tilde{x}^* + y_b^T\tilde{\dual}^*
        \end{bmatrix}}_2}
    \end{multline}
    and we expand the $2$-norm to find
    \begin{multline}
        H_{2,2}(G, C)\Expectation{\norm{\begin{bmatrix}
            -Y\tilde{x}^* + y_b\\  -Y^T\tilde{\dual}^* + z_c^0\\
            -(z_c^0)^T\tilde{x}^* + y_b^T\tilde{\dual}^*
        \end{bmatrix}}_2} =\\ H_{2,2}(G, C) \Expectation{\sqrt{(-Y\tilde{x}^* + y_b)^T(-Y\tilde{x}^* + y_b) +(-Y^T\tilde{\dual}^* + z_c^0)^T(-Y^T\tilde{\dual}^* + z_c^0) + (-(z_c^0)^T\tilde{x}^* + y_b^T\tilde{\dual}^*)^2}}.
    \end{multline}
    Since the square root is a concave function, we use Jensen's inequality to bound the expectation by
    \begin{multline}\label{eq:put_ex_here}
        H_{2,2}(G, C) \Expectation{\sqrt{(-Y\tilde{x}^* + y_b)^T(-Y\tilde{x}^* + y_b) +(-Y^T\tilde{\dual}^* + z_c^0)^T(-Y^T\tilde{\dual}^* + z_c^0) + (-(z_c^0)^T\tilde{x}^* + y_b^T\tilde{\dual})^2}}\leq \\
        H_{2,2}(G, C) \sqrt{\Expectation{(-Y\tilde{x}^* + y_b)^T(-Y\tilde{x}^* + y_b)} +\Expectation{(-Y^T\tilde{\dual}^* + z_c^0)^T(-Y^T\tilde{\dual}^* + z_c^0)} + \Expectation{(-(z_c^0)^T\tilde{x}^* + y_b^T\tilde{\dual}^*)^2}}.
    \end{multline}
    We then look to bound each of the three expectations. Starting with~$\Expectation{(-Y\tilde{x}^* + y_b)^T(-Y\tilde{x}^* + y_b)}$, we have
    \begin{equation}\label{eq:ex1_bound_max}
        \Expectation{(-Y\tilde{x}^* + y_b)^T(-Y\tilde{x}^* + y_b)} = \Expectation{\sum_{i=1}^m [y_{b_i} - (Y\tilde{x}^*)_i]^2}.
    \end{equation}
    We then expand the square to find
    \begin{equation}\label{eq:sub_expectation}
        \Expectation{\sum_{i=1}^m [y_{b_i} - (Y\tilde{x}^*)_i]^2} = \Expectation{\sum_{i=1}^m y_{b_i}^2 - 2(Y\tilde{x}^*)_iy_{b_i} + (Y\tilde{x}^*)_i^2} = \sum_{i=1}^m \Expectation{y_{b_i}^2} - 2\Expectation{(Y\tilde{x}^*)_iy_{b_i}} + \Expectation{(Y\tilde{x}^*)_i^2}.
    \end{equation}
    The mean of a truncated Laplace random variable is the midpoint of the support. Thus, since~$y_{b_i}~\sim \mathcal{L}_T(\frac{\Delta_{1}b}{\new{\alpha_b}\epsilon},[-2s_b, 0])$, it follows that~$\Expectation{y_{b_i}} = -s_b$, 
    and~$\Expectation{y_{b_i}^2} = \Var{y_{b_i}}+\Expectation{y_{b_i}}^2 = 2\left(\frac{\Delta_{1}b}{\new{\alpha_b}\epsilon}\right)^2 + s_b^2$.
    We substitute these into~\eqref{eq:sub_expectation} and expand the matrix multiplication to find
    \begin{multline}\label{eq:rep_with_max}
           \sum_{i=1}^m \Expectation{y_{b_i}^2} - 2\Expectation{(Y\tilde{x}^*)_iy_{b_i}} + \Expectation{(Y\tilde{x}^*)_i^2} =\\ 2m\left(\frac{\Delta_{1}b}{\new{\alpha_b}\epsilon}\right)^2 + ms_b^2 - 2\sum_{i=1}^m\Expectation{\sum_{j = 1}^n Y_{i, j}\tilde{x}^*_jy_{b_i}} + \sum_{i=1}^m\Expectation{\left(\sum_{j = 1}^n Y_{i, j}\tilde{x}^*_j\right)^2}.
    \end{multline}
    We then upper bound~$\tilde{x}^*_j\leq \chi$, where~$\chi$ is from~\eqref{eq:chi}, via the fact that that~$\tilde{x}^*_j\leq \max_{D\in\dbset, j\in [N]} \bar{x}_j^* = \chi$ for all~$j$. We have that~$Y_{i, j}\tilde{x}^*_jy_{b_i}\leq Y_{i, j}\chi y_{b_i}$. Since~$y_{b_i}$ is independent of~$Y_{i,j}$ and~$\chi$, we have that 
    \begin{equation}\label{eq:holy_independent}
        \Expectation{\sum_{j = 1}^n Y_{i, j}\chi y_{b_i}} = \chi\Expectation{\sum_{j = 1}^n Y_{i, j}}\Expectation{y_{b_i}} = -\chi s_b\Expectation{\sum_{j = 1}^n Y_{i, j}}.
    \end{equation}
    Thus, we substitute~\eqref{eq:holy_independent} and~$\chi$ into~\eqref{eq:rep_with_max} to obtain
    \begin{multline}\label{eq:ex1_bound_by_max}
          2m\left(\frac{\Delta_{1}b}{\new{\alpha_b}\epsilon}\right)^2 + ms_b^2 - 2\sum_{i=1}^m\Expectation{\sum_{j = 1}^n Y_{i, j}\tilde{x}^*_jy_{b_i}} + \sum_{i=1}^m\Expectation{\left(\sum_{j = 1}^n Y_{i, j}\tilde{x}^*_j\right)^2}\leq \\ 2m\left(\frac{\Delta_{1}b}{\new{\alpha_b}\epsilon}\right)^2 + ms_b^2 + 2s_b\chi\sum_{i=1}^m\Expectation{\sum_{j = 1}^n Y_{i, j}} +\chi^2\sum_{i=1}^m\Expectation{\left(\sum_{j = 1}^n Y_{i, j}\right)^2}.
    \end{multline}
    
    Note that for all indices~$i,j$ such that~$A(D)_{i,j}\neq0$ by definition we have~$Y_{i,j}\sim \mathcal{L}_T(\frac{\Delta_{1,1}A}{\new{\alpha_A}\epsilon},[0, 2s_A]),$ with each~$Y_{i,j}$ mutually independent for all~$i, j$, and~$Y_{i,j} = 0$ if~$A(D)_{i,j} = 0$. 
    As a result,~$\Expectation{\sum_{j = 1}^n Y_{i, j}} = \sum_{j = 1}^{n} \Expectation{Y_{i, j}} = n^0_i s_A$, where~$n^0_i$ is the number of non-zero elements in row~$i$ of~$Y$. Additionally, with independent~$Y_{i,j}$, we have that
    \begin{equation}\label{eq:y_ex}
        \Expectation{\left(\sum_{j = 1}^n Y_{i, j}\right)^2} = \Var{\sum_{j = 1}^n Y_{i, j}}+\Expectation{\sum_{j = 1}^n Y_{i, j}}^2 =2n^0_i\left(\frac{\Delta_{1,1}A}{\new{\alpha_A}\epsilon} \right)^2+ (n^0_i s_A)^2. 
    \end{equation}
    Substituting the first and second moments then gives us
    \begin{multline}\label{eq:ex1}
           2m\left(\frac{\Delta_{1}b}{\new{\alpha_b}\epsilon}\right)^2 + ms_b^2 + 2s_b\chi\sum_{i=1}^m\Expectation{\sum_{j = 1}^n Y_{i, j}} +\chi^2\sum_{i=1}^m\Expectation{\left(\sum_{j = 1}^n Y_{i, j}\right)^2} =\\ 2m\left(\frac{\Delta_{1}b}{\new{\alpha_b}\epsilon}\right)^2 + ms_b^2 + 2s_b\chi\sum_{i=1}^m\left(n^0_i s_A\right) +\chi^2\sum_{i=1}^m\left(2n^0_i\left(\frac{\Delta_{1,1}A}{\new{\alpha_A}\epsilon}\right)^2 + (n^0_i s_A)^2\right),
    \end{multline}
    which completes the bound on~$\Expectation{(-Y\tilde{x}^* + y_b)^T(-Y\tilde{x}^* + y_b)}$. 
    
    Next, we bound~$\Expectation{(-Y^T\tilde{\dual}^* + z_c^0)^T(-Y^T\tilde{\dual}^* + z_c^0)}$. Following a similar procedure, we have
    \begin{equation}
        \Expectation{(-Y^T\tilde{\dual}^* + z_c^0)^T(-Y^T\tilde{\dual}^* + z_c^0)} = \Expectation{\sum_{j=1}^m [z_{c_j}^0 - (Y^T\tilde{\dual}^*)_j]^2}.
    \end{equation}
    Expanding, we obtain
    \begin{multline}\label{eq:ex2_bound_max}
        \Expectation{\sum_{j=1}^n [z_{c_j}^0 - (Y^T\tilde{\dual}^*)_j]^2} = \Expectation{\sum_{j=1}^n (z_{c_j}^0)^2 - 2z_{c_j}^0(Y^T\tilde{\dual}^*)_j + (Y^T\tilde{\dual}^*)_j^2} =\\ \sum_{j=1}^n \Expectation{(z_{c_j}^0)^2} - 2\Expectation{z_{c_j}^0(Y^T\tilde{\dual}^*)_j} + \Expectation{(Y^T\tilde{\dual}^*)_j^2}.
    \end{multline}

    Since the non-zero elements of~$z_{c_j}^0$ are drawn from a (conventional, non-truncated) Laplace distribution centered at~$0$, we have that~$\Expectation{z_{c_j}^0} = 0$ and~$\Expectation{(z_{c_j}^0)^2} = \Var{z_{c_j}^0} = 2\left(\frac{\Delta_{1}c}{\new{\alpha_c}\epsilon}\right)^2$ for each~$j$ such that~$z_{c_j}^0\neq 0$. 
    For the term~$- 2\Expectation{z_{c_j}^0(Y^T\tilde{\dual}^*)_j}$, we use the law of total expectation to find
    \begin{equation}
        - 2\Expectation{z_{c_j}^0(Y^T\tilde{\dual}^*)_j} = 2\left(-\Expectation{z_{c_j}^0(Y^T\tilde{\dual}^*)_j \mid z_{c_j}\geq0}\prob{z_{c_j}\geq0} - \Expectation{z_{c_j}^0(Y^T\tilde{\dual}^*)_j \mid z_{c_j}<0}\prob{z_{c_j}<0}\right). 
    \end{equation}
    To maximize the first expectation on the right-hand side, 
    we bound~$(Y^T)_{i,j}\geq -s_A$, and for the second expectation on the right-hand side we know that~$(Y^T)_{i,j}\leq s_A$, which we substitute in to obtain
    \begin{multline}
        2\left(-\Expectation{z_{c_j}^0(Y^T\tilde{\dual}^*)_j \mid z_{c_j}\geq0}\prob{z_{c_j}\geq0} - \Expectation{z_{c_j}^0(Y^T\tilde{\dual}^*)_j \mid z_{c_j}<0}\prob{z_{c_j}<0}\right)\leq\\
        2\left(\Expectation{z_{c_j}^0(s_A \ones^{n\times m}\tilde{\dual}^*)_j \mid z_{c_j}\geq0}\prob{z_{c_j}\geq0} + \Expectation{|z_{c_j}^0|(s_A \ones^{n\times m}\tilde{\dual}^*)_j \mid z_{c_j}<0}\prob{z_{c_j}<0}\right).
    \end{multline}
    We then upper bound each term in~$\tilde{\mu}^*\leq \dualmax\ones^m$, where~$\dualmax$ is from~\eqref{eq:dualmax}, to find
    \begin{multline}
        2\left(\Expectation{z_{c_j}^0(s_A \ones^{n\times m}\tilde{\dual}^*)_j \mid z_{c_j}\geq0}\prob{z_{c_j}\geq0} + \Expectation{|z_{c_j}^0|(s_A \ones^{n\times m}\tilde{\dual}^*)_j \mid z_{c_j}<0}\prob{z_{c_j}<0}\right)\leq\\
        2\left(\Expectation{z_{c_j}^0(s_A \ones^{n\times m}\dualmax\ones^m)_j \mid z_{c_j}\geq0}\prob{z_{c_j}\geq0} + \Expectation{|z_{c_j}^0|(s_A \ones^{n\times m}\dualmax\ones^m)_j \mid z_{c_j}<0}\prob{z_{c_j}<0}\right).
    \end{multline}
    Factoring the constant terms yields
    \begin{multline}
        2\left(\Expectation{z_{c_j}^0(s_A \ones^{n\times m}\dualmax\ones^m)_j \mid z_{c_j}\geq0}\prob{z_{c_j}\geq0} + \Expectation{|z_{c_j}^0|(s_A \ones^{n\times m}\dualmax\ones^m)_j \mid z_{c_j}<0}\prob{z_{c_j}<0}\right)\leq\\
        2(s_A \ones^{n\times m}\dualmax\ones^m)_j\left(\Expectation{z_{c_j}^0 \mid z_{c_j}\geq0}\prob{z_{c_j}\geq0} + \Expectation{|z_{c_j}^0| \mid z_{c_j}<0}\prob{z_{c_j}<0}\right).
    \end{multline}
    Then,~$\Expectation{z_{c_j}^0 \mid z_{c_j}>0}\prob{z_{c_j}>0} + \Expectation{|z_{c_j}^0| \mid z_{c_j}<0}\prob{z_{c_j}<0} = \Expectation{z_{c_j}} = 0$ from the law of total expectation. Thus, we have that
    \begin{equation}\label{eq:total_ex_go_br}
        - 2\Expectation{z_{c_j}^0(Y^T\tilde{\dual}^*)_j}\leq 0.
    \end{equation}
    
    Substituting the first and second moment of~$z_{c_j}^0$, as well as~\eqref{eq:total_ex_go_br}, into~\eqref{eq:ex2_bound_max}, we find
    \begin{equation}\label{eq:bound_with_lambda}
        \sum_{j=1}^n \Expectation{(z_{c_j}^0)^2} - 2\Expectation{z_{c_j}^0(Y^T\tilde{\dual}^*)_j} + \Expectation{(Y^T\tilde{\dual}^*)_j^2} \leq2 n^{0, c} \left(\frac{\Delta_{1}c}{\new{\alpha_c}\epsilon}\right)^2 + \sum_{j=1}^n\Expectation{\left(\sum_{i = 1}^m Y_{i, j}\tilde{\dual}_i^*\right)^2}, 
    \end{equation}
    where~$n^{0, c}$ is the number of non-zero elements in~$c(D)$. From Chapter~$10$ of~\cite{arrow1958studies}, Assumption~\ref{ass:slate} implies that~$\tilde{\dual}_i^*$ is bounded by
    \begin{equation}\label{eq:dstar}
        \tilde{\dual}_i^*\leq \max_{D\in\dbset}\norm{\dual^*}_{\infty} \leq \max_{D\in\dbset}\norm{\dual^*}_{1}\leq \frac{c(D)^T\eta - c(D)^T\omega}{\min_{j\in[m]} -A(D)_j\eta+b(D)_j} = \dualmax, 
    \end{equation}
    where~$\eta$ is a solution to Problem~\eqref{opt:primal}, and~$\omega$ is any Slater point for Problem~\eqref{opt:primal}. 
    We then bound~\eqref{eq:bound_with_lambda} using~$\dualmax$ to obtain
    \begin{equation}\label{eq:ex2_sub_2}
        2n^{0, c} \left(\frac{\Delta_{1}c}{\new{\alpha_c}\epsilon}\right)^2 + \sum_{j=1}^n\Expectation{\left(\sum_{i = 1}^m Y_{i, j}\tilde{\dual}_i^*\right)^2}\leq 2n^{0, c} \left(\frac{\Delta_{1}c}{\new{\alpha_c}\epsilon}\right)^2 + \dualmax^2\sum_{j=1}^n\Expectation{\left(\sum_{i = 1}^m Y_{i, j}\right)^2}.
    \end{equation}
    From~\eqref{eq:y_ex}, we have that~$\Expectation{\left(\sum_{i = 1}^m Y_{i, j}\right)^2} = 2\sum_{i=1}^m m^0_j\left(\frac{\Delta_{1,1}A}{\new{\alpha_A}\epsilon}\right)^2 + (m^0_j s_A)^2$, 
    where~$m^0_j$ is the number of non-zero entries in column~$j$ of~$Y$,
    and we then write~\eqref{eq:ex2_sub_2} as
    \begin{equation}\label{eq:ex2}
        2n^{0, c} \left(\frac{\Delta_{1}c}{\new{\alpha_c}\epsilon}\right)^2 + \dualmax^2\sum_{j=1}^n\Expectation{\left(\sum_{i = 1}^m Y_{i, j}\right)^2} = 2n^{0, c} \left(\frac{\Delta_{1}c}{\new{\alpha_c}\epsilon}\right)^2 + \dualmax^2 m\sum_{j=1}^n \left(2m^0_j\left(\frac{\Delta_{1,1}A}{\new{\alpha_A}\epsilon}\right)^2 + (m^0_j s_A)^2\right),
    \end{equation}
    which completes the bound on~$\Expectation{(-Y^T\tilde{\dual}^* + z_c^0)^T(-Y^T\tilde{\dual}^* + z_c^0)}$. 
    
    Finally, we bound~$\Expectation{(-(z_c^0)^T\tilde{x}^* + y_b^T\tilde{\dual})^2}$. Like the previous two expectations, we begin by expanding the square:
    \begin{multline}\label{eq:ex3_simplify}
        \Expectation{(-(z_c^0)^T\tilde{x}^* + y_b^T\tilde{\dual}^*)^2} = \Expectation{((z_c^0)^T\tilde{x}^*)^2 - 2(z_c^0)^T\tilde{x}^*y_b^T\tilde{\dual}^* +  (y_b^T\tilde{\dual}^*)^2} = \\ \Expectation{((z_c^0)^T\tilde{x}^*)^2} - 2\Expectation{(z_c^0)^T\tilde{x}^*y_b^T\tilde{\dual}^*} +  \Expectation{(y_b^T\tilde{\dual}^*)^2}.
    \end{multline}
    Using a similar law of total expectation approach to~\eqref{eq:holy_independent}, we have that
    \begin{equation}
        -2\mathbb{E}[(z_c^0)^T\tilde{x}^*y_b^T\tilde{\dual}^*] \leq  0,
    \end{equation}
   which implies that
    \begin{multline}\label{eq:ex3_bound_by_max}
        \Expectation{((z_c^0)^T\tilde{x}^*)^2} - 2\mathbb{E}[(z_c^0)^T\tilde{x}^*y_b^T\tilde{\dual}^*] +  \Expectation{(y_b^T\tilde{\dual}^*)^2} \leq \Expectation{((z_c^0)^T\tilde{x}^*)^2} +  \Expectation{(y_b^T\tilde{\dual}^*)^2} =\\ \Expectation{\left(\sum_{j = 1}^n z_{c_j}^0\tilde{x}_j^*\right)^2} +  \Expectation{\left(\sum_{i = 1}^m y_{b_i}\tilde{\dual}^*_i\right)^2}.
    \end{multline}
    Using the same bounds from~\eqref{eq:ex1_bound_by_max} and~\eqref{eq:ex2_bound_max}, we bound~\eqref{eq:ex3_bound_by_max} by
    \begin{equation}
        \Expectation{\left(\sum_{j = 1}^n z_{c_j}^0\tilde{x}_j^*\right)^2} +  \Expectation{\left(\sum_{i = 1}^m y_{b_i}\tilde{\dual}^*_i\right)^2}\leq \chi^2 \Expectation{\left(\sum_{j = 1}^n z_{c_j}^0\right)^2} + \dualmax^2\Expectation{\left(\sum_{i = 1}^m y_{b_i}\right)^2}.
    \end{equation}
    Using the previously computed expectations from~\eqref{eq:rep_with_max} and~\eqref{eq:bound_with_lambda}, we find that
    \begin{equation}\label{eq:ex3}
        \chi^2 \Expectation{\left(\sum_{j = 1}^n z_{c_j}^0\right)^2} + \dualmax^2\Expectation{\left(\sum_{i = 1}^m y_{b_i}\right)^2} = 2\chi^2 n^{0, c} \left(\frac{\Delta_{1}c}{\new{\alpha_c}\epsilon}\right)^2+ \dualmax^2\left(2\left(\frac{\Delta_{1}b}{\new{\alpha_b}\epsilon}\right)^2 + s_b^2 \right), 
    \end{equation}
    which completes the bound on~$\Expectation{(-(z_c^0)^T\tilde{x}^* + y_b^T\tilde{\dual})^2}$. Substituting~\eqref{eq:ex1},~\eqref{eq:ex2}, and~\eqref{eq:ex3} into~\eqref{eq:put_ex_here} yields
    \begin{multline}\label{eq:case1}
        H_{2,2}(G, C) \sqrt{\Expectation{(-Y\tilde{x}^* + y_b)^T(-Y\tilde{x}^* + y_b)} +\Expectation{(-Y^T\tilde{\dual}^* + z_c^0)^T(-Y^T\tilde{\dual}^* + z_c^0)} + \Expectation{(-(z_c^0)^T\tilde{x}^* + y_b^T\tilde{\dual}^*)^2}}\leq \\ 
        H_{2,2}(G, C) \Bigg( 2m\left(\frac{\Delta_{1}b}{\new{\alpha_b}\epsilon}\right)^2 + ms_b^2 + 2s_b\chi\sum_{i=1}^m\left(n^0_i s_A\right) +\chi^2\sum_{i=1}^m\left(2n^0_i\left(\frac{\Delta_{1,1}A}{\new{\alpha_A}\epsilon}\right)^2 + (n^0_i s_A)^2\right) +\\
        2 n^{0, c} \left(\frac{\Delta_{1}c}{\new{\alpha_c}\epsilon}\right)^2 + m\dualmax^2\sum_{j=1}^n \left(2m^0_j\left(\frac{\Delta_{1,1}A}{\new{\alpha_A}\epsilon}\right)^2 + (m^0_j s_A)^2\right)+ 2\chi^2 n^{0, c} \left(\frac{\Delta_{1}c}{\new{\alpha_c}\epsilon}\right)^2 + \dualmax^2\left(2\left(\frac{\Delta_{1}b}{\new{\alpha_b}\epsilon}\right)^2 + s_b^2 \right) \Bigg)^{\frac{1}{2}}, 
    \end{multline}
    which completes the bound in case (i). 
    Next, we formulate the bound for case (ii). For case (ii), we seek a bound for the expression
    \begin{equation}\label{eq:bound_by_dbset}
        H_{2,2}(G, C)\Expectation{\norm{\begin{bmatrix}
            (A(D)-\tilde{A})\tilde{x}^* - (b(D)-\tilde{b})\\ (A(D)-\tilde{A})^T\tilde{\dual}^* - (c(D)-\tilde{c})\\
            (c(D)-\tilde{c})^T\tilde{x}^* - (b(D)-\tilde{b})^T\tilde{\dual}^*
        \end{bmatrix}}_2}.
    \end{equation}
    Using the triangle inequality to bound~\eqref{eq:bound_by_dbset}, we find
    \begin{multline}\label{eq:norm_stuff_here}
        H_{2,2}(G, C)\Expectation{\norm{\begin{bmatrix}
            (A(D)-\tilde{A})\tilde{x}^* - (b(D)-\tilde{b})\\ (A(D)-\tilde{A})^T\tilde{\dual}^* - (c(D)-\tilde{c})\\
            (c(D)-\tilde{c})^T\tilde{x}^* - (b(D)-\tilde{b})^T\tilde{\dual}^*
        \end{bmatrix}}_2}\leq H_{2,2}(G, C)\Bigg(\Expectation{\norm{\begin{bmatrix}
            (A(D)-\tilde{A})\tilde{x}^* \\
            (A(D)-\tilde{A})^T\tilde{\dual}^*\\
             (c(D)-\tilde{c})^T\tilde{x}^*
        \end{bmatrix}}_2}
        +\\ \Expectation{\norm{\begin{bmatrix}
            (b(D)-\tilde{b}) \\
           (c(D)-\tilde{c})\\
            (b(D)-\tilde{b})^T\tilde{\dual}^*
        \end{bmatrix}}_2
        }\Bigg). 
    \end{multline}
    For arbitrary vectors~$\alpha,\beta,\gamma\in\mathbb{R}^n$, note that
    \begin{multline}
        \norm{\begin{bmatrix}
        \alpha & \beta & \gamma
    \end{bmatrix}}_2 = \sqrt{\alpha_1^2+\cdots \alpha_n^2 +\beta_1^2+\cdots+\beta_n^2+\gamma_1^2+\cdots\gamma_n^2} = \sqrt{\norm{\alpha}_2^2+\norm{\beta}_2^2+\norm{\gamma}_2^2} \\= \norm{\begin{bmatrix}
        \norm{\alpha}_2 & \norm{\beta}_2 & \norm{\gamma}_2
    \end{bmatrix}}_2.
    \end{multline}
    Therefore, we can rewrite~\eqref{eq:norm_stuff_here} as
\begin{multline}\label{eq:bound_by_cauchy}
        H_{2,2}(G, C)\Bigg(\Expectation{\norm{\begin{bmatrix}
            (A(D)-\tilde{A})\tilde{x}^* \\
            (A(D)-\tilde{A})^T\tilde{\dual}^*\\
             (c(D)-\tilde{c})^T\tilde{x}^*
        \end{bmatrix}}_2}
        + \Expectation{\norm{\begin{bmatrix}
            (b(D)-\tilde{b}) \\
           (c(D)-\tilde{c})\\
            (b(D)-\tilde{b})^T\tilde{\dual}^*
        \end{bmatrix}}_2}\Bigg) = \\ H_{2,2}(G, C)\Bigg(\Expectation{\norm{\begin{bmatrix}
            \norm{(A(D)-\tilde{A})\tilde{x}^*}_2 \\
            \norm{(A(D)-\tilde{A})^T\tilde{\dual}^*}_2\\
             \norm{(c(D)-\tilde{c})^T\tilde{x}^*}_2
        \end{bmatrix}}_2}
        + \Expectation{\norm{\begin{bmatrix}
            \norm{(b(D)-\tilde{b})}_2 \\
           \norm{(c(D)-\tilde{c})}_2\\
            \norm{b(D)-\tilde{b}\tilde{\dual}^*}_2
        \end{bmatrix}}_2
        }\Bigg).
    \end{multline}
    Using the Cauchy-Schwarz inequality, we bound~\eqref{eq:bound_by_cauchy} as
    \begin{multline}\label{eq:sub_bounds_here}
        H_{2,2}(G, C)\Bigg(\Expectation{\norm{\begin{bmatrix}
            \norm{(A(D)-\tilde{A})\tilde{x}^*}_2 \\
            \norm{(A(D)-\tilde{A})^T\tilde{\dual}^*}_2\\
             \norm{(c(D)-\tilde{c})^T\tilde{x}^*}_2
        \end{bmatrix}}_2}
        + \Expectation{\norm{\begin{bmatrix}
            \norm{(b(D)-\tilde{b})}_2 \\
           \norm{(c(D)-\tilde{c})}_2\\
            \norm{b(D)-\tilde{b}\tilde{\dual}^*}_2
        \end{bmatrix}}_2
        }\Bigg) \leq \\ H_{2,2}(G, C)\Bigg(\Expectation{\norm{\begin{bmatrix}
            \norm{A(D)-\tilde{A}}_F\norm{\tilde{x}^*}_2 \\
            \norm{(A(D)-\tilde{A})^T}_F\norm{\tilde{\dual}^*}_2\\
             \norm{(c(D)-\tilde{c})^T}_2\norm{\tilde{x}^*}_2
        \end{bmatrix}}_2}
        + \Expectation{\norm{\begin{bmatrix}
            \norm{(b(D)-\tilde{b})}_2 \\
           \norm{(c(D)-\tilde{c})}_2\\
            \norm{b(D)-\tilde{b}}_2\norm{\tilde{\dual}^*}_2
        \end{bmatrix}}_2
        }\Bigg).
    \end{multline}
    For all~$i$ and~$j$, 
    the maximum value of~$(A(D)_{i,j} - \tilde{A}_{i,j})$ occurs when~$\tilde{A}_{i,j} = \hat{A}_{i,j}$, where~$\hat{A}_{i,j}$ is from~\eqref{eq:ahat}.
    Similarly, for~$b(D)-\tilde{b}$, the maximum value of~$(b(D)_{i} - \tilde{b}_{i})$ occurs when~$\tilde{b}_{i} = \hat{b}_{i}$, where~$\hat{b}_{i}$ is from~\eqref{eq:bhat}.
    Using these bounds and the fact
    that each~$\norm{\tilde{x}^*}_2\leq \sqrt{n}\chi$ and each~$\norm{\tilde{\mu^*}}_2\leq \sqrt{n}\dualmax$, we substitute these bounds  into~\eqref{eq:sub_bounds_here} to obtain
    \begin{multline}
        H_{2,2}(G, C)\left(\Expectation{\norm{\begin{bmatrix}
            \norm{A(D)-\tilde{A}}_F\norm{\tilde{x}^*}_2 \\
            \norm{(A(D)-\tilde{A})^T}_F\norm{\tilde{\dual}^*}_2\\
             \norm{(c(D)-\tilde{c})^T}_2\norm{\tilde{x}^*}_2
        \end{bmatrix}}_2}
        + \Expectation{\norm{\begin{bmatrix}
            \norm{(b(D)-\tilde{b})}_2 \\
           \norm{(c(D)-\tilde{c})}_2\\
            \norm{b(D)-\tilde{b}}_2\norm{\tilde{\dual}^*}_2
        \end{bmatrix}}_2
        }\right)\leq \\
        H_{2,2}(G, C)\left(\norm{\begin{bmatrix}
            \sqrt{n}\norm{(A(D)-\hat{A})}_F\chi \\
            \sqrt{m}\norm{(A(D)-\hat{A})^T}_F\dualmax\\
             \sqrt{n}\sqrt{\sum_{i=1}^n\Expectation{(c(D)_i-\tilde{c}_i)^2}}\chi
        \end{bmatrix}}_2
        + \norm{\begin{bmatrix}
            \norm{(b(D)-\hat{b})}_2 \\
           \sqrt{\sum_{i=1}^n\Expectation{(c(D)_i-\tilde{c}_i)^2}}\\
            \sqrt{m}\norm{b(D)-\hat{b}}_2\dualmax
        \end{bmatrix}}_2
        \right).
    \end{multline}
    From case (i), we know that~$\Expectation{(c(D)_i-\tilde{c}_i)^2} = 2\left(\frac{\Delta_{1}c}{\new{\alpha_c}\epsilon}\right)^2$ for all~$i$, which we substitute in for the expectation to obtain
    \begin{multline}
         H_{2,2}(G, C)\left(\norm{\begin{bmatrix}
           \sqrt{n}\norm{(A(D)-\hat{A})}_F\chi \\
            \sqrt{m}\norm{(A(D)-\hat{A})^T}_F\dualmax\\
             \sqrt{n}\sqrt{\sum_{i=1}^n\Expectation{(c(D)_i-\tilde{c}_i)^2}} \chi
        \end{bmatrix}}_2
        + \norm{\begin{bmatrix}
            \norm{(b(D)-\hat{b})}_2 \\
           \sqrt{\sum-{i=1}^n\Expectation{(c(D)_i-\tilde{c}_i)^2}}\\
            \sqrt{m}\norm{b(D)-\hat{b}}_2\dualmax
        \end{bmatrix}}_2
        \right)=\\
        H_{2,2}(G, C)\left(\norm{\begin{bmatrix}
            \sqrt{n}\norm{(A(D)-\hat{A})}_F\chi \\
            \sqrt{m}\norm{(A(D)-\hat{A})^T}_F\dualmax\\
             2\sqrt{n}\frac{\Delta_{1}c}{\new{\alpha_c}\epsilon}\chi
        \end{bmatrix}}_2
        + \norm{\begin{bmatrix}
            \norm{(b(D)-\hat{b})}_2 \\
           2\frac{\Delta_{1}c}{\new{\alpha_c}\epsilon}\\
            \sqrt{m}\norm{b(D)-\hat{b}}_2\dualmax
        \end{bmatrix}}_2
        \right), 
    \end{multline}
    which completes the bound in case (ii). Thus, we define
    \begin{equation}
    \rho = 
        \begin{cases}
            \Bigg( 2m\left(\frac{\Delta_{1}b}{\new{\alpha_b}\epsilon}\right)^2 + ms_b^2 + 2s_b\chi\sum_{i=1}^m\left(n^0_i s_A\right) +\dualmax^2\left(2\left(\frac{\Delta_{1}b}{\new{\alpha_b}\epsilon}\right)^2 + s_b^2\right)+\\\chi^2\sum_{i=1}^m\left(2n^0_i\left(\frac{\Delta_{1,1}A}{\new{\alpha_A}\epsilon}\right)^2 + (n^0_i s_A)^2\right) +
         2n^{0, c} \left(\frac{\Delta_{1}c}{\new{\alpha_c}\epsilon}\right)^2 +\\ m\dualmax^2\sum_{j=1}^n \left(2m^0_j\left(\frac{\Delta_{1,1}A}{\new{\alpha_A}\epsilon}\right)^{\new{2}} + (m^0_j s_A)^2\right)+ 2\chi^2 n^{0, c} \left(\frac{\Delta_{1}c}{\new{\alpha_c}\epsilon}\right)^2  \Bigg)^{\frac{1}{2}} & \parbox[t]{4.5cm}{\vspace{-1.7cm}if $\tilde{A}_{i,j} = A(D)_{i,j}+(s_A+Z_{i,j})\ind{A(D)\neq 0}_{i,j}$ and $\tilde{b}_i = b(D)_i-s_b+z_{b_i}$ for all $i,j$}\\

         \norm{\begin{bmatrix}
            \sqrt{n}\norm{(A(D)-\hat{A})}_F\chi \\
            \sqrt{m}\norm{(A(D)-\hat{A})^T}_F\dualmax\\
             2\sqrt{n}\frac{\Delta_{1}c}{\new{\alpha_c}\epsilon}\chi
        \end{bmatrix}}_2
        + \norm{\begin{bmatrix}
            \norm{(b(D)-\hat{b})}_2 \\
           2\frac{\Delta_{1}c}{\new{\alpha_c}\epsilon}\\
            \sqrt{m}\norm{b(D)-\hat{b}}_2\dualmax
        \end{bmatrix}}_2
         & \text{otherwise}
        \end{cases}
    \end{equation}
    and say that
    \begin{equation}
        \Expectation{\norm{\begin{bmatrix}
            x^*\\ \dual^*
        \end{bmatrix} - \begin{bmatrix}
            \tilde{x}^*\\ \tilde{\dual}^*
        \end{bmatrix}}_2}\leq H(G, C)\rho,   
    \end{equation}
    which we substitute into~\eqref{eq:sub_ex_here} to obtain
    \begin{equation}
        \Expectation{c(D)^Tx^*-c(D)^T\tilde{x}^*} \leq \norm{c(D)}_2\Expectation{\norm{\begin{bmatrix}
            x^*\\ \dual^*
        \end{bmatrix} - \begin{bmatrix}
            \tilde{x}^*\\ \tilde{\dual}^*
        \end{bmatrix}}_2}\leq \norm{c(D)}_2H(G, C)\rho, 
    \end{equation}
    which completes the proof.
    \hfill$\square$

    \new{\subsection{Proof of Corollary~\ref{cor:concentration}}\label{app:concentration}}
    
    \new{Since the feasible region of the private constraints is a subset of the feasible region of the 
    original, non-private constraints, the furthest apart that a private and non-private solution can be is the furthest distance between vertices in the original, non-private feasible region. This distance is equal to the diameter of that feasible region, which we denote~$\diam(\mathcal{F}(D))$. Additionally, the closest point to~$x^*$ in the non-private feasible region is simply~$x^*$ itself, and thus the minimum distance between two solutions is~$0$. Thus, defining~$R = \norm{x^*-\tilde{x}^*}_2$, we have that~$R\in[0, \diam(\mathcal{F}(D))]$.}
    
    \new{Hoeffding's inequality states that for a sum of~$n$ bounded random variables~$z_i\in[a, b]$, we have that
    \begin{equation}\label{eq:hoe}
        \prob{z-\Expectation{z}\leq \nu}\leq \exp\left(\frac{-2\nu^2}{n(b-a)^2}\right).
    \end{equation}
    Substituting~$a=0$,~$b = d(\mathcal{F}(D))$, and~$n=1$ into~\eqref{eq:hoe} we find
    \begin{equation}\label{eq:hoe2}
        \prob{R-\Expectation{R}\leq \nu}\leq \exp\left(\frac{-2\nu^2}{\diam(\mathcal{F}(D)^2}\right).
    \end{equation}
    Substituting
    \begin{equation}
        \nu = \diam(\mathcal{F}(D))\sqrt{\frac{\log(\frac{1}{t})}{2}}
    \end{equation}
into~\eqref{eq:hoe2} and taking the complementary probability yields
\begin{equation}
    \prob{R - \Expectation{R}\geq \diam(\mathcal{F}(D))\sqrt{\frac{\log(\frac{1}{t})}{2}}}\geq 1-t,
\end{equation}
which completes the proof. }


\hfill$\square$
    \new{\section{Additional Numerical Results}}
     \new{\subsection{Privacy Budget Analysis}}\label{subsec:priv_budget}
    \new{In this section, we explore the tradeoffs that arise when varying the privacy budget allocation between each component of an LP. Consider again the optimization problem from Section~\ref{sec:sim}:}
    \new{\begin{equation}
        \begin{aligned}
        &\begin{aligned}
            \underset{x\geq 0}{\operatorname{maximize}} &\quad \sum_{i\in[N]}\sum_{j\in[M]}p_{ij}(D)x_{ij}
        \end{aligned}
            \\
            &\begin{aligned}
                \operatorname{subject} \operatorname{to } \,\,&\sum_{j\in[M]} x_{ij}\leq n_i \quad \text{for }i\in[N]\\
                &\sum_{i\in[N]} p_{ij}(D) x_{ij} \leq b(D)_j \quad \text{for }j\in[M].
            \end{aligned}
        \end{aligned}   
        \tag{EX12}\label{opt:EX12}
    \end{equation} }
    \new{Figure~\ref{fig:example_1} illustrates the sub-optimality for an evenly distributed privacy budget~$\epsilon$, i.e.,~$\alpha_i = \frac{1}{3}$ for all~$i\in\{A, b, c\}$. Here, we evaluate the sub-optimality in the case of an unevenly distributed privacy budget. Specifically, we consider~$\alpha_c \in \{\frac{1}{3}, \frac{1}{2}, \frac{3}{4}, \frac{99}{100}\}$, and~$\alpha_A=\alpha_b = \frac{1-\alpha_c}{2}$. The baseline case of~$\alpha_A= \alpha_b=\alpha_c = \frac{1}{3}$ has~$28.25\%$ sub-optimality at~$\epsilon = 1$. As more budget is allocated to the cost, performance improves: at~$\epsilon = 1$, we see~$16.88\%$ sub-optimality when~$\frac{99}{100}$ of the privacy budget is allocated towards the cost. The improvement in performance with increasing~$\alpha_c$ implies that the most significant driving factor in sub-optimality in this scenario is the privatization of the cost. This property highlights the modularity of our method: if one component of an LP drives sub-optimality more than the others, then Algorithm~\ref{algo:solve} allows for the tuning of~$\alpha_i$ for~$i\in\{A, b, c\}$ to improve performance without changing the overall privacy guarantee afforded to~$D$.}
    \begin{figure}
        \centering
        \input{Figures/appendix_budget}
        \caption{\new{Sub-optimality gap with varying privacy budget allocated to the cost function, with the remaining cost budget divided evenly among the constraints.}}
        \label{fig:privacy_budget}
    \end{figure}
    
    \new{\subsection{Application to Constrained Markov decision processes}}\label{app:cmdp}

    \new{In this section, we present numerical simulations in the setting of constrained Markov decision processes described in~\cite{benvenuti2024guaranteed}, which we define next. We use~$\phi(S)$ to denote the set of probability distributions over a finite set~$S$.
    \begin{definition}[Constrained Markov Decision Process;~\cite{altman2021constrained}] \label{def:cmdp}
A Constrained Markov Decision Process (CMDP) is the tuple~$\mathcal{M} = (\mathcal{S}, \mathfrak{A}, r, \mathcal{T}, \mu, f, f_0)$, where $\mathcal{S}$ is the finite set of states and $\mathfrak{A}$ is the finite sets of actions, 
with~$|\mathcal{S}| = p$ and~$|\mathfrak{A}| = q$. Then $r:\mathcal{S}\times \mathfrak{A}\rightarrow \mathbb{R}$ is the reward function, $\mathcal{T}:\mathcal{S}\times \mathfrak{A}\rightarrow \simplex(\mathcal{S})$ is the transition probability function, and~$\mu\in\simplex(s)$ is a probability distribution over the initial states. Additionally, $f_i(D):\mathcal{S}\times \mathfrak{A}\to [0, f_{\text{max}, i}]$ for~$i\in[N]$ are immediate costs which depend on the database~$D$, and~$\Expectation{\sum_{t=0}^{\infty}\gamma^t f(D;s_t)}\leq f_0(D)$ with~$f_{0}(D)\in\mathbb{R}^N$  are constraints where~$f(D, s_t) = [f_1(D, s_t),\ldots, f_N(D, s_t)]^T$. 
\end{definition}
We use~$\mathcal{T}(s, a, y)$ denote the probability of transitioning from state~$s$ to state~$y$ when taking action~$a$. We consider CMDPs where constraints may be written as linear inequalities, i.e.,~$A(D)X\leq f_0(D)$, where~$A(D)\in\mathbb{R}^{pq\times N}$, $p, q, f_0$ are from Definition~\ref{def:cmdp}, and~$X$ is the decision variable in the policy synthesis optimization problem, which we describe next.}

\new{To solve an MDP is to compute an optimal policy for it, i.e., the function dictating which action to take given a state. This can be done efficiently via linear programming~\cite{puterman2014markov}:}
\new{
    \begin{equation}
        \begin{aligned}
        &\begin{aligned}
            \underset{ x_{\pi}}{\operatorname{maximize}} \quad &\sum_{s\in S}\sum_{\action\in A}&r(s, \action)x_{\pi}(s, \action)
        \end{aligned}
            \\
            &\begin{aligned}
                \text{ s.t. }   \quad &f(D; x_{\pi}(s, \action)) \leq f_0(D)\quad \text{ for all } s\in\mathcal{S}, \action\in\mathcal{A}, \\
                &\sum_{\action'\in \mathcal{A}} \ x_{\pi}(s',\action') - \gamma
                \sum_{s\in \mathcal{S}}\sum_{\action\in \mathcal{A}}x_{\pi}(s,\action)\mathcal{T}(s,\action,s') = \mu(s')\quad \text{for all } s' \in \mathcal{S}\\
                &x_{\pi}(s, \action) \geq 0 \quad \text{ for all } s\in\mathcal{S}, \action\in\mathcal{A}.
            \end{aligned}
        \end{aligned}    
        \tag{MDP-D}\label{opt:MDP-D}
    \end{equation} 
}
\new{The optimal policy~$\pi^*$ can be computed from the optimal solution to Problem~\eqref{opt:MDP-D}~$x_{\pi}^*$ via}
\new{
\begin{equation}\label{eq:pifromx}
    \pi^*(a\mid s) = \frac{x_{\pi}^*(s, a)}{\sum_{a'\in\mathfrak{A}}x_{\pi}^*(s, a')}.
\end{equation}}
\new{
Any policy for a given MDP admits a value function~$v_{\pi}$, defined as~$v_{\pi}(s) = \mathbb{E}\left[\sum_{t=0}^{\infty} \gamma^t r(s_t, a_t)\mid a_t\sim\pi(s_t), s_0 = s\right]$, and is readily computable~\cite{puterman2014markov}. Examples of constraints that~$A$ may encode include the probability of reaching a goal state or safety in the sense of avoiding hazardous states.}

\new{The constraint definition in Definition~\ref{def:cmdp} is more general than that in~\cite{benvenuti2024guaranteed} by including a dependency on a database~$D$. As a result, computing the constraints in Problem~\eqref{opt:MDP-D} may leak sensitive information about~$D$. Thus, we privatize~$f(D, x_{\pi}(s, a))$ and~$f_0(D)$ using Algorithm~\ref{algo:solve} in order to protect the privacy of~$D$.}

\new{We consider the safety scenario originally proposed in~\cite{chow2018lyapunov} and extended in~\cite{benvenuti2024guaranteed}, where~$\SH:D\to \mathcal{S}$ the mapping from a database~$D$ to a set of hazardous states~$\SH(D)$
and~$f(D; s) = \beta_{s}\mathbb{I}\{s\in\SH(D)\}$. In words,~$f$ encodes that an agent incurs penalty~$\beta_s$ for occupying state~$s$, while the database~$D$ encodes which states are hazardous. The formulation of~$f$ above yields the constraint~$\Expectation{\sum_{t=0}^{\infty}\gamma^t \beta_{s_t}f(D; s_t)\mid s_0, \pi}\leq f_0(D)$, which takes the form~$A(D)X\leq f_0(D)$, where~$A(D)$ is a row vector with
\begin{equation}
    A(D)_{i} =\begin{cases}
        \beta_{s}\gamma & \text{if }s\in \SH(D)\\
        0 & \textnormal{otherwise}
    \end{cases}.
\end{equation}
Without any protections, the use of~$A(D)$ in policy synthesis 
may leak sensitive information about~$D$, i.e., reveal knowledge that a state is hazardous, and we use Algorithm~\ref{algo:solve} to formulate a privatized form of Problem~\eqref{opt:MDP-D} that we can solve to obtain a privacy-preserving solution~$x_{\tilde{\pi}^*}$. We then post-process this solution with~\eqref{eq:pifromx} to obtain a privacy-preserving policy~$\tilde{\pi}^*$. To assess the performance of the privatized policy~$\tilde{\pi}^*$, we use the cost of privacy metric from~\cite{gohari2020privacy, benvenuti2023differentially, benvenuti2024guaranteed}, defined as
\begin{equation}
    \xi = \frac{v_{\tilde{\pi}^*}(s_0)-v_{\pi^*}(s_0)}{v_{\pi^*}(s_0)}.
\end{equation}}
\begin{figure}
    \centering
    \input{Figures/gridworld2}
    \caption{\new{Grid in which the agent starts at the blue state,
    its goal is the green state, and hazardous states are red. 
    }}
    \label{fig:safety_grid}
\end{figure}

\new{Consider the gridworld environment in Figure~\ref{fig:safety_grid}. We apply Algorithm~\ref{algo:solve} compute a privatized policy~$\tilde{\pi}^*$, which preserves the privacy of the hazardous state set and the hazard tolerance~$f_0(D)$ while guaranteeing that the safety constraint is satisfied. We set~$\beta_i = 0.6$ for all~$i\in\SH$ and we take~$\sup_{D\in\dbset} A(D)_{i,j} = 0.9$. Additionally, we take~$f_0 = 0.6$ and~$\inf_{D\in\dbset} b(D)_{i} = 0.3$.}

\new{In Figure~\ref{fig:cop} we evaluate the cost of privacy as a percent, equal to~$\xi\times 100\%$, of the policy generated using Algorithm~\ref{algo:solve}, where both~$f(D, s_t)$ and~$f_0(D)$ are privatized. Additionally, we compare against a policy computed using the approach detailed in~\cite{munoz2021private} that can only keep~$f_0(D)$ private. That approach leaks sensitive information about~$D$, 
and similar to case (ii) in Section~\ref{sec:sim} the work in~\cite{munoz2021private} provides a baseline to compare how performance is affected when privacy is applied to both~$A(D)$ and~$b(D)$ instead of only~$b(D)$. Accordingly, we expect Algorithm~\ref{algo:solve} to have a higher cost of privacy, by virtue of keeping more information private then~\cite{munoz2021private}.}
\new{
We cannot compare against~\cite{hsu2014privately} in this scenario as their approach failed to converge to a solution for all samples we simulated at every~$\epsilon$ value.  }

\new{We consider privacy parameters~$\epsilon \in[0.1, 10]$ and~$\delta = 0.1$, and budget allocations~$\alpha_A = 0.99$ and~$\alpha_b = 0.01$ averaged over~$200$ samples. 
In the strong privacy regime, i.e.,~$\epsilon= 1$, Algorithm~\ref{algo:solve} yields~$0.22\%$ sub-optimality, which is~$0.21\%$ more than only privatizing~$f_0(D)$ using the approach in~\cite{munoz2021private}. This low sub-optimality with strong privacy implies that this problem is highly compatible with privacy. Additionally, as privacy weakens we see the performance gap between Algorithm~\ref{algo:solve} and~\cite{munoz2021private} decrease, with only a~$0.11\%$ difference in sub-optimality at~$\epsilon = 3$, implying that under typical privacy implementations, we find negligible difference in performance when providing privacy for~$A(D)$ in addition to~$b(D)$ in this scenario.}
\begin{figure}
    \centering
%
%
\definecolor{chocolate2267451}{RGB}{226,74,51}
\definecolor{dimgray85}{RGB}{85,85,85}
\definecolor{gainsboro229}{RGB}{229,229,229}
\definecolor{lightgray204}{RGB}{204,204,204}
\definecolor{steelblue52138189}{RGB}{52,138,189}
\definecolor{black}{RGB}{0, 0, 0}
\definecolor{GTblue}{RGB}{0, 48, 87}
\definecolor{GTgold}{RGB}{179, 163, 105}
\definecolor{UFOrange}{RGB}{250, 70, 22}
\definecolor{UFblue}{RGB}{0, 33, 165}
\begin{tikzpicture}

\begin{axis}[%
width=0.36\figW,
height=\figH,
axis background/.style={fill=gainsboro229},
axis line style={white},
scale only axis,
xlabel=\textcolor{black}{{Privacy Strength, $\epsilon$}},
xtick style={color=dimgray85},
x grid style={white},
yminorticks=true,
y grid style={white},
ylabel=\textcolor{black}{Cost of Privacy,~$\xi\times 100\%$},
xmajorgrids,
ymajorgrids,
yminorgrids,
tick align=outside,
tick pos=left,
yticklabel={$\pgfmathprintnumber{\tick}$\%},
legend pos = north east,
]
\addplot [color=UFOrange, ultra thick]
  table[row sep=crcr]{%
0.1	0.219933557629557\\
0.2	0.220853962181431\\
0.3	0.224933943275288\\
0.4	0.220082714312181\\
0.5	0.22300558551445\\
0.6	0.219920857953001\\
0.7	0.219998462022545\\
0.8	0.222879701579695\\
0.9	0.210493742723093\\
1	0.209681395020848\\
1.1	0.213128259877276\\
1.2	0.212986674922747\\
1.3	0.209722915294403\\
1.4	0.208241855677032\\
1.5	0.210070726574992\\
1.6	0.202143152879331\\
1.7	0.179290810907531\\
1.8	0.18366588500343\\
1.9	0.163095910970824\\
2	0.15916546477635\\
2.1	0.158963177487318\\
2.2	0.152938901777344\\
2.3	0.141754538397442\\
2.4	0.138665982578708\\
2.5	0.131602328590212\\
2.6	0.12592180743745\\
2.7	0.130584461031345\\
2.8	0.124631912951658\\
2.9	0.123300093124979\\
3	0.124570336501826\\
3.1	0.125744381959872\\
3.2	0.121964245761639\\
3.3	0.12341829805833\\
3.4	0.123137840314235\\
3.5	0.121910964245015\\
3.6	0.124235545107881\\
3.7	0.12542738059027\\
3.8	0.123766813526612\\
3.9	0.122756387854092\\
4	0.121685248586413\\
4.1	0.122297825492843\\
4.2	0.123812826824359\\
4.3	0.123496974516043\\
4.4	0.121090287179134\\
4.5	0.123169540944373\\
4.6	0.119191629939363\\
4.7	0.116691650539054\\
4.8	0.122487885540746\\
4.9	0.121230085309424\\
5	0.120010383655206\\
5.1	0.11906053500527\\
5.2	0.115705707518882\\
5.3	0.12382864063314\\
5.4	0.123222475491627\\
5.5	0.114721893272035\\
5.6	0.116757294266118\\
5.7	0.121924958481683\\
5.8	0.118146683102972\\
5.9	0.121324380437135\\
6	0.120871930410489\\
6.1	0.119471634136673\\
6.2	0.113965605085602\\
6.3	0.119592493117156\\
6.4	0.115768586905511\\
6.5	0.118703533640066\\
6.6	0.116254660874962\\
6.7	0.113576385751546\\
6.8	0.117219734959583\\
6.9	0.117526287739406\\
7	0.120687493017406\\
7.1	0.119083735013437\\
7.2	0.109219126375531\\
7.3	0.111310995747117\\
7.4	0.111011837907987\\
7.5	0.108719663142318\\
7.6	0.105876912649788\\
7.7	0.105399597052475\\
7.8	0.100532284203073\\
7.9	0.0975800074763762\\
8	0.0987440351071103\\
8.1	0.0959777450422745\\
8.2	0.0938711383926269\\
8.3	0.0875920586859077\\
8.4	0.086731223664119\\
8.5	0.0866438010435774\\
8.6	0.0881286604842673\\
8.7	0.0857554240286767\\
8.8	0.0804278481967026\\
8.9	0.0774579141089072\\
9	0.0742707552800405\\
9.1	0.0766889989709513\\
9.2	0.0749656315550816\\
9.3	0.0699888545575929\\
9.4	0.0651084001878293\\
9.5	0.0618183052970905\\
9.6	0.0626763014730934\\
9.7	0.0572099660811544\\
9.8	0.0513055691263667\\
9.9	0.0518523984079438\\
10	0.0521788287674716\\
};
\addlegendentry{Algorithm~\ref{algo:solve}}

\addplot [color=UFblue, ultra thick]
  table[row sep=crcr]{%
0.1	0.00882394247002594\\
0.2	0.00886282526462909\\
0.3	0.00910368892282197\\
0.4	0.00872239377364948\\
0.5	0.00889392983489612\\
0.6	0.00883016723994517\\
0.7	0.00889532870905659\\
0.8	0.0090002811102098\\
0.9	0.00834109247082864\\
1	0.00848609442012799\\
1.1	0.00881770894135615\\
1.2	0.0086751423102659\\
1.3	0.00872082545730462\\
1.4	0.00891790147822993\\
1.5	0.00898641509523781\\
1.6	0.00886537087249277\\
1.7	0.00857649687539081\\
1.8	0.00898314006428546\\
1.9	0.00853172827716919\\
2	0.00882581522826076\\
2.1	0.00881770786457271\\
2.2	0.00903511152695003\\
2.3	0.00899793242465271\\
2.4	0.00896658337822439\\
2.5	0.00873991993056849\\
2.6	0.00881770647907511\\
2.7	0.00911786417309815\\
2.8	0.00890340197456085\\
2.9	0.00867472002165974\\
3	0.0087090302357639\\
3.1	0.00897234942393713\\
3.2	0.00872237992254457\\
3.3	0.00877004615502261\\
3.4	0.00873480660782666\\
3.5	0.00870745646989875\\
3.6	0.00885762656402942\\
3.7	0.00908563018040314\\
3.8	0.00894608924988022\\
3.9	0.00884120245767309\\
4	0.00886537166461004\\
4.1	0.00896069650950503\\
4.2	0.0089438920659768\\
4.3	0.00892862587746672\\
4.4	0.00893865010917528\\
4.5	0.00902884177152854\\
4.6	0.00862907616161885\\
4.7	0.00852436642065786\\
4.8	0.00896070772685499\\
4.9	0.00900836043512676\\
5	0.0088970037792245\\
5.1	0.00874935649969154\\
5.2	0.00857939911524097\\
5.3	0.00902208469110837\\
5.4	0.00908329360227443\\
5.5	0.00842231347557872\\
5.6	0.00853173051858248\\
5.7	0.00896069631599476\\
5.8	0.0087223831196404\\
5.9	0.00891303625229295\\
6	0.00891303927135146\\
6.1	0.00872238274357326\\
6.2	0.00843640078802731\\
6.3	0.00877004463805058\\
6.4	0.00857503515434774\\
6.5	0.00877004460490686\\
6.6	0.00853909692295628\\
6.7	0.00843640589834384\\
6.8	0.00872238430978291\\
6.9	0.00891303303560565\\
7	0.0092220964691395\\
7.1	0.00900836120513379\\
7.2	0.00848764472559032\\
7.3	0.00881899251655703\\
7.4	0.00898894332154823\\
7.5	0.00877004679325418\\
7.6	0.008865374735438\\
7.7	0.00897668848120202\\
7.8	0.00872238446824078\\
7.9	0.00877004538542271\\
8	0.00896069867288066\\
8.1	0.00873115204671993\\
8.2	0.00871051721310983\\
8.3	0.00862025455605973\\
8.4	0.00867678527290478\\
8.5	0.00866835919210934\\
8.6	0.00915135433265532\\
8.7	0.00872238703185556\\
8.8	0.00873165929527334\\
8.9	0.00857940063406011\\
9	0.00854933220306064\\
9.1	0.00882084720237688\\
9.2	0.00896070590742248\\
9.3	0.00886536975734038\\
9.4	0.00857099684711104\\
9.5	0.00850813554363019\\
9.6	0.00896069769233231\\
9.7	0.008722387159527\\
9.8	0.00853172731406312\\
9.9	0.00872238738110045\\
10	0.00913412329534952\\
};
\addlegendentry{\cite{munoz2021private}}

\end{axis}
\end{tikzpicture}%
    \caption{\new{Cost of Privacy for privately generated policies using Algorithm~\ref{algo:solve} and~\cite{munoz2021private}. The sub-optimality in~\cite{munoz2021private} remains constant, however their approach leaks private information about~$D$ since their approach cannot privatize~$A(D)$. The sub-optimality of Algorithm~\ref{algo:solve} approaches that of~\cite{munoz2021private} as~$\epsilon$ increases, indicating that at low privacy we recover the performance of~\cite{munoz2021private} without privacy leakage.
    }}
    \label{fig:cop}
\end{figure}

\new{\section{Reproducibility Checklist}\label{ap:checklist}
Some authors of this work maintain an affiliation which restricts the sharing of code, and thus to comply with these guidelines, we refrain from sharing the code used to generate the numerical results seen in the paper and technical appendix. We stress that the implementation of Algorithm~\ref{algo:solve} in the settings of Section~\ref{sec:sim} and Technical Appendix~\ref{app:cmdp} is straightforward. The optimization problems are made in the \texttt{problem} environment in MATLAB, and the constraint terms~$A$ and~$b$ and the cost vector~$c$ are directly accessed in this environment and modified as described in Algorithm~\ref{algo:solve} before solving using any solver. Snippets of code approved for sharing with the public may be made available upon request.}

\samepage

\end{document}